\documentclass[a4paper]{amsart}
\usepackage{url}
\usepackage{amsmath}
\usepackage[makeroom]{cancel}

\usepackage{graphicx}
\usepackage[all]{xy}
\usepackage[mathscr]{eucal}
\usepackage{amsmath,amssymb,amsfonts}

\newtheorem{theorem}{Theorem}[section]
\newtheorem{lemma}[theorem]{Lemma}

\newtheorem{example}[theorem]{Example}
\newtheorem{examples}[theorem]{Examples}
\newtheorem{proposition}[theorem]{Proposition}
\newtheorem{remark}[theorem]{Remark}

\newtheorem{question}[theorem]{Question}
\newtheorem{terminology}[theorem]{Terminology}
\newtheorem{definition}[theorem]{Definition}
\newtheorem{terminologyandnotation}[theorem]{Terminology and notation}

\usepackage{stackengine}
\usepackage{xcolor}
\usepackage[all]{xy}

\providecommand{\keywords}[1]{\textbf{Keywords---} #1}

\title{Finitary Cartesian closed varieties and semigroup actions}

\author{Mark V. Lawson\,
}


\address{Mark V. Lawson, Department of Mathematics,
Maxwell Institute for Mathematical Sciences,
Heriot-Watt University,
Riccarton,
Edinburgh EH14 4AS,
UNITED KINGDOM}
\email{m.v.lawson@hw.ac.uk}

\begin{document}
\dedicatory{This paper is dedicated to the memory of Philip Scott, friend and colleague.}

\begin{abstract}
We build on some ideas of Richard Garner.
Let $M$ be a monoid and $B$ a Boolean algebra.
A matched pair $[B|M]$ consists of $B$ and $M$ satisfying some axioms introduced by Garner.
He then showed that every such matched pair determines
a Boolean left restriction monoid $S = S[B|M]$.
In this paper, we show that the data of a $[B|M]$-set, an action of a matched pair,
may be encoded by means of what we call a factorizable Boolean supported action by $S$.
This means that the category of $[B|M]$-sets is equivalent to
the category of factorizable Boolean supported actions of  $S$.
We deduce, as a result of Garner's work, that every non-degenerate finitary Cartesian closed variety 
is equivalent to the category of factorizable Boolean supported actions of some Boolean left restriction monoid $S$.
\end{abstract}


\maketitle

\section{Introduction}

This paper builds on some excellent work of Richard Garner  \cite{G2024, G2025}.
He proved two things.
First, he provided a syntactic characterization of when a variety\footnote{We use the word `variety' in the sense of universal algebra.
In this paper, we shall only deal with finitary varieties; that is, varieties in which all operations have a finite arity. Garner also dealt with infinitary varieties which we shall not consider here.}
was Cartesian closed \cite[Theorem 5.5]{G2024}.
Second, he provided a semantic characterization of when a variety was Cartesian closed.
He did this in terms of what he called `actions' (the quotes mean something here) of matched pairs \cite[Proposition 7.13]{G2024};
a matched pair consists of a Boolean algebra and a monoid and certain operations (which we shall describe later).
In \cite[Proposition 3.3, Proposition 3.4, Theorem 3.5]{G2025}, he showed that matched pairs could be used to construct (what we term; see below)
Boolean left restriction monoids.
Thus his theory combines category theory, universal algebra and semigroup theory.
Our paper will focus on the semigroup theory aspect of his work.
The natural question is: if Cartesian closed varieties are equivalent to `actions' of matched pairs,
can we rewrite the latter as honest-to-goodness actions (without the quotes) of Boolean left restriction monoids?
The goal of this paper is to answer this question in the affirmative.
The only real problem we faced was reconciling the universal algebra with the category theory:
this accounts for the frequent appearance of the empty set. 

Before we explain what we actually do, we describe our point of view; this will serve to orientate the reader.
Semigroup theory is, of course, defined as the study of associative binary operations.
In practice, however, it can often be regarded as the {\em algebraic} study of functions, to be contrasted with
category theory and the theory of computable functions.
The subtitle of \cite{JS} describes what we have in mind.
Just as with category theory, the word `function' should be interpreted broadly.
Thus, inverse semigroup theory is the abstract theory of {\em partial} bijections.
Similarly, the left restriction monoids, which play a starring role in this paper,
should be viewed as the abstract theory of partial functions along with their domains of definition;
to be precise, you should think in terms of partial functions as being composed from left-to-right.
In addition, we shall need a class of left actions of such monoids.
As we shall see, you can regard such actions as being on presheaves of sets. 

We shall use a little bit of category theory in this paper for which the references are \cite{BW} and \cite{CWM}.
We recall some important definitions here.
Let $F \colon \mathcal{A} \rightarrow \mathcal{B}$ be a functor between categories.
We say that $F$ is {\em full} when, for every pair of objects $A$ and $B$, where $F(A) = A'$ and $F(B) = B'$,
and every arrow $f' \colon A' \rightarrow B'$ in $\mathcal{B}$
there exists an arrow $f \colon A \rightarrow B$ for which $F(f) = f'$.
We say that $F$ is {\em faithful} if whenever $f_{1}, f_{2} \colon A \rightarrow B$ such that $F(f_{1}) = F(f_{2})$
then $f_{1} = f_{2}$.
We say that $F$ is {\em essentially surjective} if for each object $A'$ in $\mathcal{B}$ there is an object $A$ in $\mathcal{A}$
such that $A'$ is isomorphic to $F(A)$.
The categories $\mathcal{A}$ and $\mathcal{B}$ are {\em equivalent} iff there is a functor from $\mathcal{A}$
to $\mathcal{B}$ which is full, faithful and essentially surjective: see,
\cite[Theorem 1, Section 4, Chapter IV]{CWM}.
We shall also use Cartesian closed categories,
but rather than the definition to be found in  \cite{L1988}
we shall use the definition in  \cite{BW}:
a {\em Cartesian closed category} is a category with a  terminal object,
binary products and exponentials.
We now define these terms.
If $A$ is an object in a category then $1_{A}$ denotes the identity morphism on $A$.
In a category, a {\em terminal object} is one for which there is a unique morphism from every object to the terminal object.
A category has {\em binary products} if for each pair of objects $A$ and $B$
there is an object $A \,\Box\, B$ together with morphisms $\pi_{1} \colon A \, \Box \, B \rightarrow A$ and 
$\pi_{2} \colon A \, \Box \, B \rightarrow B$ such that for any object $C$ and morphisms $\alpha \colon C \rightarrow A$ and 
$\beta \colon C \rightarrow B$ there is a unique morphism $\gamma \colon C \rightarrow A \, \Box \, B$
such that $\pi_{1} \gamma = \alpha$ and $\pi_{2} \gamma = \beta$.
A category has {\em exponentials} if for every pair of objects $A$ and $B$ there is an object $B^{A}$ 
and a morphism $\mathbf{eval} \colon B^{A} \, \Box \, A \rightarrow B$ such that for any morphism $\phi \colon C \, \Box \,A \rightarrow B$
there is a unique morphism $\hat{\phi} \colon C \rightarrow B^{A}$ such that\footnote{The element $\hat{\phi}$ is usually denoted by $\lambda \phi$.} 
$\phi = \mathbf{eval} (\hat{\phi} \, \Box  \,1_{A})$.
Despite not using the definition to be found in  \cite{L1988},
I think Phil would have been amused to see a  connection between Cartesian closed categories and semigroup theory.

We shall also need {\em initial objects} which are defined dually to terminal objects.

One piece of notation we should highlight.
If $\theta \colon A \rightarrow B$ is a function and $A' \subseteq A$ then
the {\em restriction of $\theta$ to $A'$} will be denoted by $\theta | A'$.

In Section~4, we prove that the category of Boolean supported actions of $S$,
where $S$  is a Boolean left restriction monoid,
is equivalent to the category 
$[\mathsf{Proj}(S)|\mathsf{Tot}(S)]$-{\bf Set}
(we shall explain all terminology and notation later).
However, Garner \cite{G2024, G2025} proved that 
every non-degenerate finitary Cartesian closed variety is equivalent to one of the form {\bf $[B|M]$-Set}
where $B$ is a Boolean algebra and $M$ is a monoid.
It will therefore follow from our work that every non-degenerate finitary Cartesian closed variety 
is equivalent to a category consisting of a single Boolean left restriction monoid $S$
and its Boolean supported actions.

The paper proceeds as follows.
In Section~2,
we describe a class of left restriction monoids, that we call `factorizable',
in terms of what we call `weak matched pairs' $[E|M]$, where $E$ is a meet semilattice with a top element
and $M$ is a monoid. 
It will turn out that Boolean left restriction monoids are always factorizable.
In Section~3,
we introduce a class of actions of left restriction monoids that we term `supported';
this extends the \'etale actions, defined in \cite{S2011}, from inverse monoids to left restriction monoids.
We also describe a category of `actions' of $[E|M]$, denoted by $[E|M]$-{\bf Set}.
We then fix a left restriction monoid $S$ and prove that 
a class of supported actions of $S$, that we also term `factorizable', is equivalent to the category $[E|M]$-{\bf Set}.
In Section~4, we show how the above theory works out when we specialize to  Boolean
left restriction monoids and their Boolean supported actions; it is here that we meet Garner's work again.
Thus we proceed from the general to the specific.

Why should you care about all of this? The groups of units of what are termed `\'etale Boolean left restriction monoids'
are topological full groups of the `ample groupoids' within the theory of $C^{\ast}$-algebras.
Such groups can be seen as generalizations of the Thompson groups, which were originally constructed from the $\lambda$-calculus
(whose categorical incarnation is the theory of Cartesian closed categories).
The fact that there is a connection between the Thompson groups and the $\lambda$-calculus results from a throwaway line on page 1 of \cite{Higman}
and some comments by Statman in \cite[Section 3]{Statman}.

\begin{terminologyandnotation}
{\em The word `semilattice' will mean a meet semilattice with a top element denoted by 1.
The meet operation will be denoted by concatenation.
The usual order on a semilattice will be denoted by $\leq$.
We shall use the word `homomorphism' to mean any structure-preserving function.
Thus homomorphisms of monoids will be required to preserve the identity
and homomorphisms of semigroups-with-zero will be required to preserve the zero.
I shall use `iff' as an abbreviation for `if and only if' throughout this paper.}
\end{terminologyandnotation}

\noindent
{\bf Acknowledgements }The author would like to thank the Isaac Newton Institute for Mathematical Sciences, Cambridge,
for support and hospitality during the programme {\em Topological groupoids and $C^{\ast}$-algebras},
where work on this paper was undertaken.
This work was supported by EPSRC grant EP/V521929/1.
The author would particularly like to thank
Ying-Fen Lin for the invitation
as well as the other organizers
Astrid an Huef,
Nadia Larsen,
Xin Li,
and
Stuart White.

I would like to thank the anonymous referee for a careful reading of the first version of this paper
and many constructive comments.
I would also like to thank Richard Garner with whom I had an email exchange about nothing ($\varnothing$) which was very important.
Needless to say, any errors or infelicities that remain are the fault of the author alone.\\

\section{The structure of factorizable  left restriction monoids}

The goal of this section is to describe the structure of what we call `factorizable left restriction monoids'
in terms of semilattices and monoids generalizing, slightly, some ideas to be found in \cite{G2025}.
It will transpire that the Boolean left restriction monoids (the monoids we are really interested in)
are automatically factorizable in our sense.

Define a monoid $S$ to be a {\em left restriction monoid} if
it is equipped with a unary operation $s \mapsto s^{+}$, called {\em plus}, satisfying the following axioms:
\begin{itemize}
\item[{\rm (LR1).}] $(s^{+})^{+} = s^{+}$. 
\item[{\rm (LR2).}] $(s^{+}t^{+})^{+} = s^{+}t^{+}$.
\item[{\rm (LR3).}] $s^{+}t^{+} = t^{+}s^{+}$.
\item[{\rm (LR4).}] $s^{+}s = s$.
\item[{\rm (LR5).}] $(st)^{+} = (st^{+})^{+}$.
\item[{\rm (LR6).}] $st^{+} = (st)^{+}s$.
\end{itemize}
See my paper \cite{L2025} for the origins of these axioms (or, rather, their left-right duals)
and the basic properties that follow from them.
Denote by $\mathsf{Proj}(S)$ those elements $a \in S$ such that $a^{+} = a$, called {\em projections}.
We are working in a monoid so it is important to note that  $1^{+} = 1$.
In the case where our monoid has a zero, then $0^{+} = 0$.
The set of projections is a meet semilattice (under multiplication) with top element (in other words, it is a monoid under meet).

\begin{remark}{\em Axiom (LR4) is where the `left' in `left restriction monoid' comes from.
 Axiom (LR5) enables us, in certain circumstances, to replace an element by a projection.
 Axiom (LR6) will do a lot of the heavy lifting for us in this paper.}
 \end{remark}

Those elements $a \in S$ such that $a^{+} = 1$ are called  {\em total elements};
the set of total elements of $S$ is denoted by $\mathsf{Tot}(S)$.
This is a monoid by axiom (LR5).

Define a partial order on a left restriction monoid by $s \leq t$ iff $s = s^{+}t$.
We call this the {\em natural partial order} and it will be the only order we shall consider
on a left restriction monoid.
It can be shown that $s \leq t$ iff $s = et$ where $e$ is a projection.
Although it appears that we have chosen sides,
we have not.
By applying axiom (LR6), we have that $te= (te)^{+}t$
whenever $e$ is a projection.
Thus $s \leq t$ iff $s = tf$ for some projection $f$.
If $s \leq t$ then $s^{+} \leq t^{+}$;
we shall use this result many times.

Let $S$ and $T$ be left restriction monoids.
A {\em homomorphism} $\theta \colon S \rightarrow T$ of left restriction monoids
is a monoid homomorphism from $S$ to $T$ such that $\theta (a^{+}) = \theta (a)^{+}$ for all $a \in S$.
The proof of the following lemma is immediate from the definition of homomorphism, and will play an important role later.

\begin{lemma}\label{lem:maps} Let $\theta \colon S \rightarrow T$ be a homomorphism of left restriction monoids.
Then $\theta$ maps $\mathsf{Tot}(S)$ to $\mathsf{Tot}(T)$
and $\mathsf{Proj}(S)$ to  $\mathsf{Proj}(T)$.
\end{lemma}

We now make an important  definition.
Let $M$ be a monoid and $E$ a meet semilattice with a top element.
We shall say that the ordered pair $(E,M)$
is a {\em weak matched pair},
denoted by $[E|M]$,
if the following axioms are satisfied:
\begin{itemize}
\item[{\rm (MP1).}] The monoid $M$ acts on the set $E$. We denote this action by $(m,e) \mapsto m \ast e$.
\item[{\rm (MP2).}] $m \ast 1 = 1$ for all $m \in M$.
\item[{\rm (MP3).}] $m \ast (ef) = (m \ast e)(m \ast f)$.
\item[{\rm (MP4).}] For each $e \in E$, there is a right congruence on $M$ denoted by $\equiv_{e}$; 
the $\equiv_{e}$-class containing the element $m$ will be denoted by $[m]_{e}$.
\item[{\rm (MP5).}] $\equiv_{1}$ is the equality right congruence.
\item[{\rm (MP6).}] If $f \leq e$ in $E$ then $m \equiv_{e} n$ implies that $m \equiv_{f} n$.
\item[{\rm (MP7).}] If $a \equiv_{e} a'$ then $ma \equiv_{m \ast e} ma'$ for any $m \in M$.
\item[{\rm (MP8).}] If $a \equiv_{e} a'$ then $e(a \ast f) = e(a' \ast f)$ for any $f \in E$.
\end{itemize}
The rationale for this definition is the following result.

\begin{lemma}\label{lem:putty}
The left restriction monoid $S$ gives rise to the weak matched pair 
$$[\mathsf{Proj}(S)|\mathsf{Tot}(S)]$$
when we define 
$$m \ast e = (me)^{+}, \text{ and }
m \equiv_{e} n \Leftrightarrow em = en,$$
where $e \in \mathsf{Proj}(S)$ and $m,n \in \mathsf{Tot}(S)$.
\end{lemma}
\begin{proof} Although our situation is slightly more general, the proof of the claims above is very similar to the proof given
by Garner in \cite[Proposition 3.3]{G2025}.
For that reason, we shall just sketch out our proof here.
We have already seen that $M = \mathsf{Tot}(S)$ is a monoid and $E = \mathsf{Proj}(S)$ is 
a meet semilattice with a top element from the definition of left restriction monoid. 
We now use the properties of a left restriction monoid, axioms (LR1)--(LR6), to show that the axioms
for a weak matched pair,
axioms (MP1)--(MP8),
really do hold.
Axiom (MP1) holds: use (LR5) to establish the associativity of the action.
Axiom (MP2) holds from the very definition of $M$.
Axiom (MP3) holds: use (LR2), (LR5) and (LR6).
Axiom (MP4) holds: this is immediate from the definition.
Axiom (MP5) holds: immediate from the properties of the identity.
Axiom (MP6) holds: we are given that $f \leq e$ and $em = en$. Thus $fm = fn$.
Axiom (MP7) holds: use (LR6).
Axiom (MP8) holds: use (LR2) and (LR5).
\end{proof}

Let $[E|M]$ and $[E'|M']$ be weak matched pairs.
A {\em homomorphism} from $[E|M]$ to $[E'|M']$ is a pair of maps $(\alpha, \beta)$
where both $\alpha \colon M \rightarrow M'$ 
and
$\beta \colon E \rightarrow E'$ are homomorphisms,
such that 
$$\beta (m \ast e) = \alpha (m) \ast \beta (e), \text{ and } 
m \equiv_{e} n \Rightarrow \alpha (m) \equiv_{\beta (e)} \alpha (n).$$
The weak  matched pairs form a category.
If the matched pairs $[E|M]$ and $[E'|M']$ are {\em isomorphic}, then we write  $[E|M] \cong [E'|M']$.

\begin{lemma}\label{lem:functor} 
There is a functor from the category of left restriction monoids to the category
of weak matched pairs.
\end{lemma}
\begin{proof} If $S$ is a left restriction monoid, then $[\mathsf{Proj}(S)|\mathsf{Tot}(S)]$
is a weak matched pair by  Lemma~\ref{lem:putty}.
Let $\theta \colon S \rightarrow T$ be a homomorphism between left restriction monoids.
By Lemma~\ref{lem:maps}, 
there are homomorphisms 
$\theta | \mathsf{Tot}(S) \colon \mathsf{Tot}(S) \rightarrow \mathsf{Tot}(T)$
and
$\theta | \mathsf{Proj}(S) \colon \mathsf{Proj}(S) \rightarrow \mathsf{Proj}(T)$
which together constitute a homomorphism of weak matched pairs.
\end{proof}

We say that a left restriction monoid $S$ is {\em factorizable} if for each $s \in S$ there is a total element $\hat{s} \in \mathsf{Tot}(S)$ such that $s \leq \hat{s}$.
Factorizable left restriction monoids form a full subcategory of the category of left restriction monoids and their homomorphisms
since homomorphisms preserve the order and total elements.
There is, therefore, an inclusion functor from the category of factorizable left restriction monoids and homomorphisms to the category of left restriction monoids and homomorphisms.
What follows is a mild generalization of \cite[Proposition 3.3, Proposition 3.4]{G2025}.

\begin{lemma}\label{lem:fiennes}\mbox{}
\begin{enumerate}
\item Let $[E|M]$ be a weak matched pair. Then we can construct a factorizable left restriction monoid $S = S[E|M]$
such that $[E|M] \cong [\mathsf{Proj}(S)|\mathsf{Tot}(S)]$.
\item There is a functor from the category of weak matched pairs to the category of factorizable left restriction monoids.
\end{enumerate}
\end{lemma}
\begin{proof}
(1)  We shall only sketch the proof since it is similar to the one given by Garner in \cite{G2025}.
The elements of $S = S[E|M]$ are ordered pairs $(e, [m]_{e})$ where $e \in E$ and $m \in M$.
Define
$$(e,[a]_{e})(f,[b]_{f}) = (e(a \ast f), [ab]_{e(a \ast f)}).$$
This is well-defined by (MP4), (MP6), (MP7) and (MP8).
We therefore have a binary operation on $S$.
This is associative by (MP1) and (MP3) and has an identity $(1,[1]_{1})$ by (MP1) and (MP2).
It is therefore a monoid.
Define 
$$(e,[a]_{e})^{+} = (e,[1]_{e}).$$
It is now routine to check that $S$ is a left restriction monoid.
The total elements are those of the form $(1,[a]_{1})$,
the projections are the elements of the form $(e,[1]_{e})$,
and we have that $(e,[a]_{e}) \leq (f,[b]_{f})$ iff $e \leq f$ and $a \equiv_{e} b$.
From this, it follows that $S$ is factorizable since $(e,[a]_{e}) \leq (1,[a]_{1})$.
Define $\alpha \colon M \rightarrow \mathsf{Tot}(S)$ by $a \mapsto (1,[a]_{1})$,
a bijection by (MP5),
and define $\beta \colon E \rightarrow \mathsf{Proj}(S)$ by $e \mapsto (e,[1]_{e})$, also a bijection.
It is now routine to check that $(\alpha, \beta)$ defines an isomorphism $[E|M] \cong [\mathsf{Proj}(S)|\mathsf{Tot}(S)]$.

(2) Let $(\alpha, \beta)$ be a homomorphism from the weak matched pair  $[E|M]$ to  the weak matched pair $[E'|M']$.
Define $\theta \colon S[E|M] \rightarrow S[E'|M']$ by 
$$(e, [a]_{e}) \mapsto (\beta (e), [\alpha (a)]_{\beta (e)}).$$
This is a homomorphism of left restriction monoids. 
\end{proof}

Lemma~\ref{lem:netta} and Theorem~\ref{them:outhwaite} below follow up a suggestion of the referee.

\begin{lemma}\label{lem:netta} The category of factorizable left restriction monoids
is a co\-reflective subcategory of the category of left restriction monoids.
\end{lemma}
\begin{proof}
Let $S$ be a left restriction monoid.
Put $\mathsf{F}(S) = \mathsf{Tot}(S)^{\downarrow}$, the elements of $S$ below a total element in the natural partial order.
Then $\mathsf{F}(S)$ is a factorizable left restriction monoid, and it contains all the projections of $S$. 
It defines a functor since any homomorphism between left restriction monoids preserves elements beneath total elements.
There is an inclusion of $\mathsf{F}(S)$ into $S$.
Let $\theta \colon T \rightarrow S$ be a homomorphism of a factorizable left restriction monoid $T$
to a left restriction monoid $S$.
Then $T$ maps to $\mathsf{F}(S)$ uniquely.
Thus every homomorphism $\theta \colon T \rightarrow S$, with $T$ factorizable, factors uniquely through the inclusion of
$\mathsf{F}(S)$ in $S$.
Hence, the inclusion of the category of factorizable left restriction monoids in the category of left restriction monoids has the functor $\mathsf{F}$ as a right adjoint.
\end{proof}

We now come to our first main result.

\begin{theorem}\label{them:outhwaite}
The category of factorizable left restriction monoids is equivalent to the category of weak matched pairs. 
\end{theorem}
\begin{proof} 
We prove that the restriction of the functor, described in Lemma~\ref{lem:functor}, to factorizable left restriction monoids is faithful, full and essentially surjective.

It is easy to show that this functor is faithful since each element $s$ in a factorizable  left restriction monoid can be written $s = s^{+}\hat{s}$
where $s^{+}$ is a projection and $\hat{s}$ is a total element.
Thus two homomorphisms which agree on the projections and the total elements must agree everywhere.

To prove that the functor is full, let $S$ and $T$ be factorizable left restriction monoids
and let 
$$(\alpha, \beta) \colon  [\mathsf{Proj}(S)|\mathsf{Tot}(S)] \rightarrow [\mathsf{Proj}(T)|\mathsf{Tot}(T)]$$
be a homomorphism of weak matched pairs.
Let $ s \in S$, where $s = s^{+}\hat{s}$ for some total element $\hat{s}$.  
Define a map $\theta \colon S \rightarrow T$  
by $\theta (s) = \beta (s^{+}) \alpha (\hat{s})$.
We have to check that this map is really well-defined.
However, if, also, $s = s^{+}s'$, where $s'$ is a total element, then $s' \equiv_{s^{+}} \hat{s}$ 
and so $\alpha (s') \equiv_{\beta (s^{+})} \alpha (\hat{s})$.
It follows that $\beta (s^{+})\alpha (s') = \beta (s^{+}) \alpha (\hat{s})$, and so $\theta$ is indeed well-defined.

We now prove that $\theta$ is a homomorphism.
Suppose that $s = s^{+}\hat{s}$ and $t = t^{+}\hat{t}$,
where $\hat{s}$ and $\hat{t}$ are total elements.
From $s \leq \hat{s}$ and $t \leq \hat{t}$, we have that $st \leq \hat{s}\hat{t}$.
But $\hat{s}\hat{t}$ is a total element.
Thus 
$\theta (s) = \beta (s^{+})\alpha (\hat{s})$,
$\theta (t) = \beta (t^{+})\alpha (\hat{t})$,
and 
$$\theta (st) = \beta ((st)^{+}) \alpha (\hat{s}\hat{t}).$$
On the other hand
$$\theta (s) \theta (t) =  \beta (s^{+})\alpha (\hat{s})\beta (t^{+})\alpha (\hat{t}).$$
By axiom (LR6), we have that $\alpha (\hat{s})\beta (t^{+}) = [ \alpha (\hat{s})\beta (t^{+})]^{+}\alpha  (\hat{s})$.
It follows that
$$\theta (s) \theta (t) =  \beta (s^{+})
[ \alpha (\hat{s})\beta (t^{+})]^{+}\alpha  (\hat{s})
\alpha (\hat{t}).$$
But $\beta ((st)^{+}) = \beta ((st^{+})^{+})$ by axiom (LR5).
Thus
$$\beta ((st)^{+}) 
= \beta ((s^{+} \hat{s}t^{+})^{+}) 
=  \beta (s^{+}) \beta (\hat{s} \ast t^{+})
= \beta (s^{+}) \alpha (\hat{s}) \ast \beta (t^{+}),$$
where we have used the definition of a homomorphism between weak matched pairs.
However, if we use the definition of $\ast$, we get that
$$\beta (s^{+}) \alpha (\hat{s}) \ast \beta (t^{+})
= 
\beta (s^{+}) [\alpha (\hat{s}) \beta (t^{+})]^{+}.$$ 
We have therefore proved that
$$\theta (s)\theta (t) = \beta ((st)^{+})\alpha (\hat{s}) \alpha (\hat{t})$$
and $\alpha$ is itself a homomorphism.
We have therefore proved that $\theta$ is a homomorphism.
It is straightforward to check that $\theta$ is actually a homomorphism of left restriction monoids.
Thus our functor is full.

The fact that this functor is essentially surjective now follows from part (1) of Lemma~\ref{lem:fiennes}.
\end{proof}

What follows is tangential to this paper, but is of some interest.
Theorem~\ref{them:outhwaite} can be viewed as a generalization of \cite{EF}.
Specifically, recall that an inverse monoid is said to be {\em factorizable} if every element is beneath a unit in the natural partial order.
This implies that $S = \mathsf{E}(S)\mathsf{U}(S)$; that is, every element can be written as the product of an idempotent and a unit.
Just as with left restriction monoids, if $s = eg$ then $e = \mathbf{r}(s) = ss^{-1}$.
Thus the element $e$ is determined by the element $s$ itself.
Define an equivalence relation $\equiv_{e}$ on $\mathsf{U}(S)$ by $g \equiv_{e} h$ iff $eg = eh$ where $e$ is an idempotent.
This implies that if $s = eg = eh$ then $g \equiv_{\mathbf{r}(s)} h$.
Because $\equiv_{e}$ is a right congruence on a group, it is determined by a subgroup.
This subgroup is denoted by $G_{e}$ in \cite{EF}.
This suggests that such inverse monoids can be constructed from groups, $G$, and meet semilattices, $E$.
Accordingly, we define weak matched pairs $[E|G]$ as before except that we require $G$ to be a group and not merely a monoid.
We may form the left restriction monoid $S[E|G]$ except that this time the total elements form a group.
The following result ensures that apart from replacing the monoid by a group (which, after all, is a special kind of monoid) we do not have to do anything else to get an inverse monoid.

\begin{lemma} 
Every factorizable left restriction monoid in which the total elements form a group is a factorizable  inverse monoid.
\end{lemma}
\begin{proof} Let $S$ be a factorizable left restriction monoid in which the total elements form a group.
Let $s \in S$.
By assumption $s \leq g$ where $g$ is a total element.
Also, by assumption, $g$ is a unit.
Thus $s = s^{+}g$ where $s^{+}$ is a projection and $g$ is a unit.
Put $t = g^{-1}s^{+}$.
It is routine to check that $sts = s$.  
This shows that the semigroup is regular.
We now show that every idempotent is a projection.
Let $a$ be an idempotent.
By assumption, as before, we may write $a = a^{+}g$, where $g$ is a unit. 
Since $a$ is an idempotent, we have that $a^{+}ga^{+}g = a^{+}g$.
Thus  $a^{+}ga^{+} = a^{+}$.
But $a^{+}g = a$.
Thus $aa^{+} = a^{+}$.
This shows that $a^{+} \leq a$ since $a^{+}$ is equal to $a$ times a projection.
But then $a^{+} = a^{+}a$,
from the definition of the natural partial order.
We deduce that $a = a^{+}$,
and so every idempotent is a projection.
However, the projections commute.
Thus the idempotents commute.
But a regular semigroup in which the idempotents commute is inverse by \cite[Theorem 1.2.14]{Lawson}.
Factorizability is then automatic from the assumption that $S$ is a factorizable left restriction monoid.
\end{proof}

\section{Supported actions}

In this section, we shall define an associated class of actions by left restriction monoids $S$
and describe them in terms of `actions' by the weak matched pair $[\mathsf{Proj}(S)|\mathsf{Tot}(S)]$.

Let $S$ be a left restriction monoid.
A left action of $S$ on a set $X$, denoted by $(s,x) \mapsto s \cdot x$, is said to be {\em supported}
if there is a function $p \colon X \rightarrow \mathsf{Proj}(S)$, called the {\em support}, such that the following two axioms hold:
\begin{itemize}
\item[{\rm (E1).}] $p(x) \cdot x = x$. 
\item[{\rm (E2).}] $p(s \cdot x) = (sp(x))^{+}$.
\end{itemize}
This definition is modelled on the definition of \'etale  action of an inverse semigroup
to be found in \cite{F2010} and \cite{S2011}. 
We have used exactly the same definition as in the inverse case,
but extended it in the obvious way to left restriction monoids.
We prefer the term `supported' here since the word `\'etale' seems to be overused;
in this, we have followed \cite{LR2021}. 
If the left restriction monoid $S$ admits a supported action on the set $X$ with support $p$ then we write $(S,X,p)$.

\begin{examples}\label{exs:examples}
{\em Here are three examples of supported actions.
Let $S$ be a left restriction monoid.
\begin{enumerate}

\item  Put $E = \mathsf{Proj}(S)$.
Define $p \colon E \rightarrow \mathsf{Proj}(S)$ to be the identity function.
Define $S \times E \rightarrow E$ by $(s,e) \mapsto s \cdot e = (se)^{+}$.
It is easy to check that this is a supported action.

\item If $e$ is any projection, then $(S,Se,p_{e})$ where $p_{e}(x) = x^{+}$
is a supported action when the action is defined via multiplication.

\item There is a unique function $p \colon \varnothing \rightarrow \mathsf{Proj}(S)$.
There is also a unique function $S \times \varnothing \rightarrow \varnothing$;
this defines an action of $S$ on $\varnothing$.
By default, $(S,\varnothing,p)$ is a supported action.
This example is only included for categorical reasons.

\end{enumerate}
}
\end{examples}

Just as in the inverse case, we can define an order $\leq$ on $X$ by $x \leq y$ iff $x = p(x) \cdot y$.
It is easy to check that $x \leq y$ iff $x = e \cdot y$, where $e$ is some projection.

\begin{lemma} \label{lem:daphne}
Let $(S,X,p)$ be a supported action of a left restriction monoid $S$.
\begin{enumerate}
\item If $x \leq y$ and $p(x) = p(y)$ then $x = y$.
\item $p$ is order preserving.
\item If $x \leq y$ in $X$ and $s \leq t$ in $S$ then $s \cdot x \leq t \cdot y$ in $X$.
\end{enumerate}
\end{lemma}
\begin{proof} (1) By definition, $x \leq y$ means that  $x = p(x) \cdot y$.
By assumption, $p(x) = p(y)$ and the result now follows by (E1).

(2) Suppose that $x \leq y$.
Then $x = p(x) \cdot y$.
By (E2), we have that 
$$p(x) = p(p(x) \cdot y) = (p(x)p(y))^{+} = p(x)p(y),$$
since the product of projections is a projection.
Thus $p(x) \leq p(y)$.

(3) By definition, $s = s^{+}t$ and $x = p(x) \cdot y$.
We calculate $p(s \cdot x)t \cdot y = (sp(x))^{+}t \cdot y$.
But $s = s^{+}t$ and so
$(sp(x))^{+}t \cdot y = (s^{+}tp(x))^{+}t \cdot y$.
By axiom (LR5) and axiom (LR2), we have that 
$(s^{+}tp(x))^{+}t \cdot y
=
s^{+}(tp(x))^{+}t \cdot y$.
We can now apply axiom (LR6), to get
$s^{+}(tp(x))^{+}t \cdot y = s^{+}tp(x) \cdot y$.
This simplifies to  $s^{+}t \cdot x$ which is just $s \cdot x$.
We have therefore proved that $s \cdot x \leq t \cdot y$.
\end{proof}

If $x,y \in X$, we write $x \approx y$ if $p(x) \cdot y = p(y) \cdot x$.
Observe that if $x,y \leq z$ then $x \approx y$.

\begin{definition}
{\em Let $(S,X,p)$ be a supported action.
For any projection $e$, we write $X_{e}$ for the set of all elements $x \in X$ such that $p(x) = e$.
Particularly significant to us will be the set $X_{1}$.}
\end{definition}

Supported actions are really just actions on presheaves as we now explain.

\begin{remark}{\em Suppose that we have a supported action $(S,X,p)$.
Then $X$ carries the structure of a presheaf of sets over the semilattice $\mathsf{Proj}(S)$ as follows:
if $f$ and $e$ are projections such that $f \leq e$ define $\phi^{e}_{f} \colon X_{e} \rightarrow X_{f}$ by $x \mapsto f \cdot x$;
observe that $\phi^{e}_{e}$ is the identity function on $X_{e}$ and if $g \leq f \leq e$ then $\phi^{e}_{g} = \phi^{f}_{g}\phi^{e}_{f}$.}
\end{remark}

\begin{lemma}\label{lem:empty} Let $(S,X,p)$,  where $X \neq \varnothing$,  be a supported action of a left restriction monoid $S$.
If $X_{1}$ is not the empty set, then $p$ is surjective.
\end{lemma}
\begin{proof} Let $e$ be any projection.
By assumption, there exists an $x \in X$ such that $p(x) = 1$.
Consider the element $e \cdot x \in X$.
From the definition of supported actions, we have that $p(e \cdot x)= (ep(x))^{+} = e$. 
Thus $p$ is surjective.
\end{proof}

We have defined when a left restriction monoid is factorizable.
We now use the same term but this time for actions.
All supported actions $(S,\varnothing, p)$ are factorizable. 
If $(S,X,p)$ is a supported action, where $X \neq \varnothing$,
then we say it is {\em factorizable} if for each $y \in X$ there exists $x \in X_{1}$ such that $y \leq x$.

Let $(S,X,p)$ and $(S,X',p')$ be supported actions of a left restriction monoid $S$.
A {\em homomorphism} of such actions is a function $\theta \colon X \rightarrow X'$ 
which satisfies two conditions:
\begin{itemize}
\item[(HSA1).] $\theta$ is {\em $S$-equivariant}, which means that $\theta (s \cdot x) = s \cdot \theta (x)$.
\item[(HSA2).] $p' \theta = p$ 
\end{itemize}
Denote by $\mbox{hom}(X,X')$ the set of all such homomorphisms.

Although not strictly necessary, the following answers an obvious question,
well-known in the inverse case \cite{S2011}.

\begin{lemma}\label{lem:red}
Let $(S,X,p)$ and $(S,X',p')$ be supported actions of a left restriction monoid $S$
where both $X$ and $X'$ are non-empty.
A homomorphism $\theta \colon X \rightarrow X'$ is order-preserving.
 \end{lemma}
\begin{proof} Suppose that $x \leq y$ in $X$.
Then $x = p(x) \cdot y$.
It follows from the definition of a homomorphism that 
$\theta (x) = \theta (p(x) \cdot y) = p(x) \cdot  \theta (y)$.
Thus $\theta (x) \leq \theta (y)$, as claimed.
\end{proof}

\begin{definition}
{\em 
Fix a left restriction monoid $S$.
Denote the category of all supported actions of $S$ by $S$-{\bf Sup},
and denote the category of all factorizable supported actions of $S$ by $S$-{\bf FSup}.
Observe that $S$-{\bf FSup} is a full subcategory of $S$-{\bf Sup}.
}
\end{definition}

We say that an initial object in a category is {\em strict} if every morphism to it is an isomorphism.

\begin{lemma}\label{lem:degaulle}
Let $S$ be a left restriction monoid.
The category $S$-{\bf Sup} has a strict initial object;
in fact, the initial object is unique.
\end{lemma}
\begin{proof}
The initial object of the category is $(S,\varnothing,p)$,
as described in part (3) of Examples~\ref{exs:examples}.
Suppose that $(S,Y,q)$ is any object of this category which admits a homomorphism
to this initial object.
This can only happen if $Y$ is itself empty.
Thus $Y$ is the initial object.
The uniqueness of the initial object is now immediate.
\end{proof}

\begin{definition}{\em 
We shall denote the unique function from the initial object of $S$-{\bf Sup}
to any object $A$ by $\zeta_{A}$.}
\end{definition}

\begin{proposition}\label{prop:sophie}
Let $S$ be a left restriction monoid.
Then the category $S$-{\bf Sup} has 
a terminal object, and binary products, which we denote by $\Box$.
\end{proposition}
\begin{proof} We shall denote the category $S$-{\bf Sup} by $\mathscr{C}$, for convenience.

We prove that $(S,E,p)$, defined in part (1) of 
Examples~\ref{exs:examples},
is the terminal object in $\mathscr{C}$. 
Let $(S,X,q)$ be any supported action.
Define $\theta \colon X \rightarrow E$ by $\theta (x) = q(x)$.
Then $\theta$ is a homomorphism of supported actions
because
$$\theta (s \cdot x) = q(s \cdot x) = (sq(x))^{+} = s \cdot \theta (x)
\text{ and }
p\theta (x)= q(x).$$
We now show that $\theta$ is the unique homomorphism with these properties.
Suppose that $\theta' \colon X \rightarrow E$ is a homomorphism.
Then $p\theta' (x) = q(x) = \theta (x)$.
But $p$ is the identity.
Thus $\theta' = \theta$. 

We show that the category $\mathscr{C}$ has binary products.
We first assume that all actions are non-initial.
Let $(S,X,p)$ and $(S,Y,q)$ be supported actions of $S$.
Define $Z = X \Box \, Y$ to be the set of all ordered pairs $(x,y)$, where $x \in X$ and $y \in Y$ such that $p(x) = q(y)$. 
Define $r \colon Z \rightarrow \mathsf{Proj}(S)$ by $r(x,y) = p(x) = q(y)$ if $(x,y) \in Z$.
Define a function $S \times Z \rightarrow Z$ by
$s \cdot (x,y) = (s \cdot x, s \cdot y)$ where $s \in S$ and $(x,y) \in Z$;
if $p(x) = q(y)$ then $p(s \cdot x) = q(s \cdot y)$ so the function is well-defined.
Observe that $(st) \cdot (x,y) = s \cdot (t \cdot (x,y))$,
and
$r(x,y) \cdot (x,y) = (x,y)$ 
and
 $r(s \cdot (x,y))  = (sr(x,y))^{+}$.
We therefore have a supported action $(S,Z,r)$.
Define $\pi_{1} \colon X \,\Box \, Y \rightarrow X$ by $\pi_{1}(x,y) = x$, the projection onto the first component.
Observe that $\pi_{1}(s \cdot (x,y)) = s \cdot \pi_{1}(x,y)$
and that $p\pi_{1} (x,y) = r(x,y)$.
Thus $\pi_{1}$ is a homomorphism.
Similarly, we may define $\pi_{2}$, the projection onto the second component, which is also a homomorphism.
We can now prove that we have defined a binary product.
Let $(S,W,t)$ be any supported action
for which there are homomorphisms 
$\alpha \colon W \rightarrow X$ 
and
$\beta \colon W \rightarrow Y$.
Define $\gamma \colon W \rightarrow  X \,\Box \, Y$
by $\gamma (w) = (\alpha (w), \beta (w))$.
The map  $\gamma$ is well-defined since $p\alpha (w) = t(w) = q\beta (w)$.
Thus $(\alpha (w), \beta (w)) \in X \,\Box \, Y$.
By definition, 
$$\gamma (s \cdot w) = (\alpha (s \cdot w), \beta (s \cdot w)) = s \cdot \gamma (w)$$
using the fact that both $\alpha$ and $\beta$ are homomorphisms.
We also have that $r \gamma (w) = t(w)$,
because  $\alpha$ is a homomorphism.
We have therefore proved that $\gamma$ is a homomorphism.
We now prove that $\gamma$ is unique with the stated properties.
Let $\gamma'$ be a homomorphism from $W$ to $X \,\Box \, Y$
such that $\pi_{1} \gamma' = \alpha$ and $\pi_{2} \gamma' = \beta$.
From the very definitions of $\pi_{1}$ and $\pi_{2}$,
we have that $\gamma' = \gamma$.

We now deal with the case where some of the actions involved are initial.
It is enough to deal with the case where $X \Box \, Y$ is initial.
Since it is initial, there are unique maps to $X$ and $Y$.
If $W$, as above, were non-initial then for any $x \in W$ we would have that $p(\alpha (x)) = t(x)$ and $q (\beta (x)) = t(x)$
which would imply that $X \Box \, Y$ would be non-empty and so would not be the initial object.
Thus $W$ is initial itself.
Then there are unique maps from $W$ to $X$ and $Y$ and a unique map from $W$ to $X \Box \, Y$.

It follows that in all cases, the supported action $X \,\Box \, Y$
equipped with the maps $\pi_{1} \colon X \,\Box \, Y \rightarrow X$
and 
$\pi_{2} \colon X \,\Box \, Y \rightarrow Y$
is the product of $(S,X,p)$ and $(S,Y,q)$ in the category $\mathscr{C}$.
\end{proof}

With reference to Proposition~\ref{prop:sophie}, we have the following.

\begin{lemma}\label{lem:gop}
Let $(S,X,p)$ and $(S,Y,q)$ be  supported actions of $S$ where $X \,\Box \, Y$
is not the empty set.
If $(x,y), (x',y') \in X \,\Box \, Y$
then $(x,y) \approx (x',y')$ iff $x \approx x'$ and $y \approx y'$.
\end{lemma}
\begin{proof}
Let $(x,y), (x',y') \in X \,\Box \, Y$.
Suppose that $(x,y) \approx (x',y')$.
By definition, $r(x,y) \cdot (x',y') = r(x',y') \cdot (x,y)$.
Thus $p(x) \cdot x' = p(x') \cdot x$ and $q(y) \cdot y' = q(y') \cdot y$.
It follows that $x \approx x'$ and $y \approx y'$.
These calculations can be run in reverse so that
$(x,y) \approx (x',y')$ iff $x \approx x'$ and $y \approx y'$.
\end{proof}

\begin{lemma}\label{lem:palmer} Let $S$ be a left restriction monoid.
Let $(S,X,p)$ and $(S,Y,q)$ be supported actions of $S$ where neither of the sets $X$ and $Y$ is empty.
Then there is a projection $e$ such that $Se\, \Box \, X$ is non-empty,
and there is a homomorphism from $Se\, \Box \, X$ to $Y$.
\end{lemma}
\begin{proof} Since the set $X$ is non-empty, we may choose $x' \in X$.
The set $Y$ is also non-empty so we can choose $y' \in Y$.
Put $e = p(x')q(y')$ and define $y_{0} = e \cdot y'$ and $x_{0} = e \cdot x'$.
We prove first that the set  $Se\, \Box \, X$ is non-empty.
Observe that $p(x_{0}) = e$.
Thus $(e,x_{0}) \in Se\, \Box \, X$.
Define $\theta \colon Se\, \Box \, X \rightarrow Y$
by $\theta (a,x) = a \cdot y_{0}$.
(HSA1) holds:
Let $s \in S$.
Then, by definition, $\theta (sa, s \cdot x) = (sa) \cdot y_{0}$,
whereas $s \cdot \theta (a,x) = s \cdot (a \cdot y_{0})$.
These two elements are equal.
(HSA2) holds: because 
$$q\theta (a, x) = q(a \cdot y_{0}) = (aq(y_{0}))^{+} = (ae)^{+} = a^{+} = r(a,x).$$
 \end{proof}

We shall now say something deeper about the category $S$-{\bf Sup}.

\begin{theorem}\label{them:CCV1} Let $S$ be a fixed left restriction monoid.
The category $\mathscr{C} = S$-{\bf Sup} is Cartesian closed.
\end{theorem}
\begin{proof} 
We refer to Proposition~\ref{prop:sophie} where we proved that this category has a terminal object and binary products.
Observe that (below) we shall always use `$\cdot$' to denote the action of a left restriction monoid $S$,
although what $S$ is actually acting on will depend on the context.
Denote the set of all {\em partial} functions from $A$ to $B$ by $\mathsf{PF}(A,B)$.

We first define what the exponential object should be before proving that it has all the properties it should have.\\

\underline{We shall deal first with the case where all objects are non-initial.}\\

Let  $(S,X,p)$ and $(S,Y,q)$ be supported actions.
By part (2) of Examples~\ref{exs:examples}, there is a supported action of $S$ on $Se$ for each projection $e$.
Define
$$Y^{X} = \bigcup_{e \in \mathsf{Proj}(S)} \{e\} \times  \mbox{hom}(Se \, \Box \, X, Y)
\text{ and } 
\pi \colon Y^{X} \rightarrow \mathsf{Proj}(S)
\text{ by } \pi ((e,\theta)) = e $$
and define a function $S \times Y^{X} \rightarrow Y^{X}$ as follows
$$(a \cdot (e,\theta)) = ((ae)^{+}, a \cdot \theta)
\text{ where }
a \cdot \theta (t,x) = \theta (ta, (ta)^{+} \cdot x).$$

\noindent
{\bf Notation.} To ease notation, we shall usually write elements of $Y^{X}$ simply as $\theta$ when we know that
$\theta \in \mbox{hom}(Se \, \Box \, X, Y)$.
With this assumption in place, we can regard $Y^{X}$ as a subset of  $\mathsf{PF}(S \times X, Y)$.\\

\begin{enumerate}
\item {\em $a \cdot \theta$ is well-defined.}
Let $\theta \in \mbox{hom}(Se\, \Box \, X, Y)$.
If $(t,x) \in S(ae)^{+}\, \Box X$ then $p_{(ae)^{+}}(t) = p(x)$.
This means that $t^{+} = p(x)$, and $t(ae)^{+} = t$ since $t \in S(ae)^{+}$.
However, 
$$p((ta)^{+} \cdot x) = ((ta)^{+}p(x))^{+} = ((ta)^{+}t^{+})^{+} = (ta)^{+}t^{+}$$
by axiom (LR2).
But $(ta)^{+}t^{+}  = (ta)^{+}$.
Thus $p((ta)^{+} \cdot x) = (ta)^{+}$.
By axiom (LR6), we have that $ta = t(ae)^{+}a = tae$.
We have shown that  $(ta, (ta)^{+} \cdot x) \in Se\, \Box \, X$.
Thus  $\theta (ta, (ta)^{+} \cdot x)$ is defined.

(HSA1) holds: We have that  $a \cdot \theta \colon S(ae)^{+}\, \Box X \rightarrow Y$.
Let $s \in S$.
We calculate $(a \cdot \theta)(s t, s \cdot x)$.
This is equal to $\theta (s (ta), s (ta)^{+} \cdot x)$,
where we have used axiom (LR6).
The result now follows from the fact that $\theta$ is $S$-equivariant and the definition of the action.

(HSA2) holds:
We have that $q (a \cdot \theta)(t,x') = (ta)^{+}$, where we have used the fact that $\theta$ is a homomorphism. 
But $r(t,x') = t^{+}$.
We are given that $t \in S(ae)^{+}$.
Thus $t(ae)^{+} = t$.
It follows that
$$(ta)^{+} = (t(ae)^{+}a)^{+} = (t(ae)^{+}a^{+})^{+} = (t(ae)^{+})^{+}$$
using (LR5) and the fact that from $ae \leq a$ we obtain $(ae)^{+} \leq a^{+}$.
Thus $(ta)^{+} = t^{+}$.

\item {\em $(S,Y^{X},\pi)$ is a supported action.}

We check first that we really do have a monoid action.
By definition, 
$(b \cdot (a \cdot \theta))(t,x) = \theta (tba, (tba)^{+} (tb)^{+} \cdot x)$.
but $(tba)^{+} \leq (tb)^{+}$
since $(tba)^{+} = (tba^{+})^{+}$ by axiom (LR5)
and $tba^{+} \leq tb$.
Thus 
$(b \cdot (a \cdot \theta))(t,x) = \theta (tba, (tba)^{+}  \cdot x)$.
On the other hand,
$((ba) \cdot \theta)(t,x) = \theta (tba, (tba)^{+}  \cdot x)$.
We have therefore proved that $(b \cdot (a \cdot \theta)) = (ba) \cdot \theta$.
It is easy to check that $1 \cdot \theta = \theta$.
(E1) holds: From the definitions $\pi (\theta) \cdot \theta = \theta$.
(E2) holds: We defined $\pi$ in order that $\pi (a \cdot \theta) = (a\pi (\theta))^{+}$.
\end{enumerate}

Let $(S,Z,r)$  be any supported action,
and let $g \colon Z\, \Box \, X \rightarrow Y$ be a homomorphism.
To define $\hat{g} \colon Z \rightarrow Y^{X}$,
we have to define $\hat{g}(z)$ for each $z \in Z$.
Accordingly, define
$$\hat{g}(z) \colon Sr(z) \, \Box \, X \rightarrow Y
\text{ by }
\hat{g}(z)(a,x) = g(a \cdot z, (ar(z))^{+} \cdot x).$$

\begin{enumerate}

\item {\em $\hat{g}(z)$ is a well-defined homomorphism}.
Observe that
if $(a,x) \in Sr(z)\, \Box \, X$ then, in fact, 
$$(a \cdot z, (ar(z))^{+} \cdot x) \in Z \, \Box \, X.$$
So, at the very least, $\hat{g}(z)$ is a well-defined function.
We now show that $\hat{g}(z)$ is a homomorphism from $Sr(z) \, \Box \, X$ to $Y$.

(HSA1) holds: Let $(t,x) \in Sr(z)\, \Box \, X$.
Then $tr(z) = t$ and $t^{+} = p(x)$.
Let $s \in S$.
Then $\hat{g}(z)(s t, s \cdot x) = g (s t \cdot z, (s t r(z))^{+} \cdot s \cdot x)$.
But by axiom (LR6), we have that $(s t r(z))^{+} s = s (tr(z))^{+}$.
Thus $\hat{g}(z)(s t, s \cdot x) = g(s t \cdot z, s (tr(z))^{+} \cdot x)$.
However, $g$ is a homomorphism and so
$g(s t \cdot z, s (tr(z))^{+} \cdot x) = s \cdot g(t \cdot z, (tr(z))^{+} \cdot x)$.

(HSA2) holds:
We now calculate
$$q\hat{g}(z)(a,x) = qg(a \cdot z, (ar(z))^{+} \cdot x) =  p((ar(z))^{+} \cdot x) = (ar(z))^{+} = a^{+}.$$

\item {\em $\hat{g}$ is a homomorphism from $Z$ to $Y^{X}$.}

(HSA1) holds: Using the definitions, 
it can be checked that
$$\hat{g}(m \cdot z) = m \cdot \hat{g}(z).$$

(HSA2) holds:
We have to prove that $\pi \hat{g} = r$.
But this is immediate from the way we have defined $\hat{g}(z)$.
\end{enumerate}

Define
$$\mathbf{eval} \colon Y^{X} \, \Box \, X \rightarrow Y \text{ by }
\mathbf{eval}(\theta, x) = \theta (\pi (\theta),x).$$
This makes sense, since to say that $(\theta, x) \in Y^{X} \, \Box \, X$
means precisely that $\pi (\theta) = p(x)$
where $\theta \colon S\pi (\theta)\, \Box \, X \rightarrow Y$.\\

\noindent
{\em $\mathbf{eval} \colon Y^{X}\, \Box \, X \rightarrow Y$ is a homomorphism. }

(HSA1) holds: Let $\theta \in Y^{X}$, where  $\theta \colon Se \, \Box \, X \rightarrow Y$,
and let $x \in X$ be such that $e = p(x)$.
Observe that $(s \cdot \theta, s \cdot x) \in Y^{X} \, \Box \, X$.
We may therefore calculate $\mathbf{eval}(s \cdot \theta, s \cdot x)$.
By definition, this is equal to $(s \cdot \theta)((s\pi(\theta))^{+}, s \cdot x)$
and so, from the definition of $s \cdot \theta$, this is equal to
$\theta ((se)^{+}s, [(se)^{+}s]^{+}s \cdot x)$.
By axiom (LR6), we can write $(se)^{+}s = se$.
Thus
$$\mathbf{eval}(s \cdot \theta, s \cdot x) = \theta (se,  se \cdot x) = s \cdot \theta (e,x) = s \cdot \theta (\pi (\theta),x),$$
where we have used the fact that $\theta$ is a homomorphism.
But this is equal to $s \cdot \mathbf{eval}(\theta, x)$.

(HSA2) holds: This follows by a direct computation.\\

\noindent
{\em The `uniqueness' part of the definition of Cartesian closed. }
If $(z,x) \in Z \, \Box \, X$
then $\mathbf{eval}(\hat{g}(z), x) = g(z,x)$, by simply applying the definitions we have made.
To prove uniqueness, we have to show that if
$\bar{g} \colon Z \rightarrow Y^{X}$ is a homomorphism
such that 
$$\mathbf{eval}(\bar{g}(z),x) = g(z,x)$$
for all $(z,x) \in Z \, \Box \, X$
then $\hat{g} = \bar{g}$.
Since $\bar{g}$ is a homomorphism, we must have that
$\pi(\bar{g}(z)) = r(z)$.
We are therefore given that
$$\bar{g}(z)(r(z),x) = \hat{g}(z)(r(z), x)$$
for any $x \in X$ such that $p(x) = r(z)$
where $z \in Z$.
Observe that if $a \in Sr(z)$ and $a^{+} = p(x)$,
then 
$$\hat{g}(z)(a,x) = (a \cdot \hat{g}(z))(a^{+},  x) = \hat{g}(a \cdot z)(a^{+}, x) = \hat{g}(a \cdot z)(a^{+}, x).$$
But $a^{+} = (ar(z))^{+} = r (a \cdot z)$ and $p(x) = a^{+}$.
Thus 
$$\hat{g}(z)(a,x) 
= \hat{g}(a \cdot z)(r(a \cdot z), (ar(z))^{+} \cdot x).$$
By assumption
$$\hat{g}(a \cdot z)(r(a \cdot z), (ar(z))^{+} \cdot x) = \bar{g}(a \cdot z)(r(a \cdot z), (ar(z))^{+} \cdot x).$$
We now put the above calculations into reverse but this time using $\bar{g}$ which has all the requisite properties.
In this way, we can show that
$$\bar{g}(a \cdot z)(r(a \cdot z), (ar(z))^{+} \cdot x)
= 
\bar{g}(z)(a,x).$$
Thus 
$$\hat{g}(z)(a,x) = \bar{g}(z)(a,x)$$
which implies that for any $z \in Z$, we have that
$\hat{g}(z)= \bar{g}(z)$.
Thus 
$\hat{g}= \bar{g}$,
as claimed. This proves uniqueness.\\

\underline{We now deal with the case where some of the objects involved are initial.}\\

We shall use the fact, proved in Lemma~\ref{lem:degaulle},  that the initial object is strict.
We deal first with the very definition of $Y^{X}$.
\begin{enumerate}

\item {\em $X$ is not initial  and $Y$ is not initial.}
By Lemma~\ref{lem:palmer}, the set $\mbox{hom}(Se \, \Box \, X,Y)$ is non-empty for some projection $e$.
Thus $Y^{X}$ is not initial.

\item {\em $X$ is initial and $Y$ is not initial.}
Then $Se \, \Box \, X$ is initial.
In this case, the set $Y^{X}$ reduces to the set $\{(e,\zeta_{Y}) \colon e \in \mathsf{Proj}(S)\}$.
The function $\pi$ does the following: $\pi (e,\zeta_{Y}) = e$.
Observe that $a \cdot (e,\zeta_{Y}) = ((ae)^{+},\zeta_{Y})$.
It follows that in this case, $(S,Y^{X}, \pi)$ is isomorphic (as expected) to the terminal object in the category.

\item {\em $X$ is not initial and $Y$ is initial.}
For each projection $e$, the set $\mbox{hom}(Se \, \Box X,Y)$ is only non-empty when $Se \, \Box X$ is initial.
We are interested in the cases where $Se \, \Box X$ is initial.
Let $e$ be a projection.
We shall say that $e$ is `bad' if there is no $a \in Se$ and $x \in X$ such that $a^{+} = p(x)$.
Denote the set of bad projections by $\mathsf{B}$.
Then $Y^{X} = \{(e,\zeta_{Y}) \colon e \in \mathsf{B}\}$.
We shall use the definition below.

\item {\em $X$ initial and $Y$ initial.} In this case, $Y^{X}$ is isomorphic to the terminal object (as expected).
\end{enumerate}

We have proved that $Y^{X}$ is an object of the category in each case.
We now have to prove that this object has the requisite universal properties in each case.

\begin{enumerate}

\item  {\em $X$ is not initial and $Y$ is not initial.}
There are two cases.

\begin{enumerate}

\item  {\em $Y^{X}\, \Box \, X$ is initial.}
There can only be a homomorphism from $Z\, \Box \, X$ to $Y^{X}\, \Box \, X$ 
if $Z\, \Box \, X$ is initial, in which case  $Z\, \Box \, X = Y^{X}\, \Box \, X$.

\item  {\em $Y^{X}\, \Box \, X$ is not initial.}
If $Z \, \Box \, X$ is initial then there is a unique homomorphism to $Y^{X} \, \Box \, X$
and a unique homomorphism to $Y$.
If $Z \, \Box \, X$ is not initial then we are back to the case discussed in the first part of the proof.
\end{enumerate}

\item {\em $X$ is initial and $Y$ is not initial.}
Both  $Y^{X} \, \Box \, X$ and $Z \, \Box \, X$ are initial
and so are equal.
  
\item {\em $X$ is not initial and $Y$ is initial.}
We prove that $Y^{X} \, \Box \, X$ is initial.
Suppose that $Y^{X} \, \Box \, X$ were not initial.
Then there would be an element $((e,\zeta_{Y}), x) \in Y^{X} \, \Box \, X$
for some projection $e$.
This implies that $e$ is bad.
But $e = p(x)$ and $e \in Se$, and so $e$ is not bad.
This contradiction implies that  $Y^{X} \, \Box \, X$ is initial.
There will only be a homomorphism from $Z\, \Box \, X$ to $Y^{X} \, \Box \, X$
if $Z \, \Box \, X$ is initial, in which case $Z \, \Box \, X = Y^{X} \, \Box \, X$.

\item {\em $X$ is initial and $Y$ is initial.}
In this case $Y^{X} \, \Box \, X$ is initial.
There will only be a homomorphism from $Z\, \Box \, X$ to $Y^{X} \, \Box \, X$
if $Z\, \Box \, X$ is initial, in which case $Z\, \Box \, X = Y^{X} \, \Box \, X$.
\end{enumerate}

In all the cases above, the axioms for a Cartesian closed category hold.
\end{proof}


We saw in Theorem~\ref{them:CCV1} that $Y^{X} \subseteq \mathsf{PF}(S \times X, Y)$.
The right hand side has a natural order because it is a set of partial functions.
We now relate this to the order in $Y^{X}$ as a supported $S$-set.

\begin{lemma}\label{lem:kate} Let $(S,X,p)$ and $(S,Y,q)$ be supported actions.
If $\theta, \phi \in Y^{X}$, where $\theta \in \mbox{\rm hom}(Se\, \Box\, X,Y)$ and 
$\phi \in \mbox{\rm hom}(Sf\, \Box\, X,Y)$,
then $\theta \leq \phi$ iff $e \leq f$ and $\theta \subseteq \phi$.
\end{lemma}
 \begin{proof} Suppose that $\theta \leq \phi$ in $Y^{X}$.
 By definition,  $\theta = e \cdot \phi$.
Using $\pi$ and the fact that we are dealing with supported actions,
we have that $e \leq f$.
We are given that $\theta = e \cdot \phi$.
Thus, for any $(t,x) \in Se \, \Box \, X$, 
we have that $\phi (t,x) = \phi (te, (te)^{+} \cdot x) = (e \cdot \phi)(t,x)=  \theta (t,x)$.
We have shown that $\theta (t,x) = \phi (t,x)$.
Conversely, suppose that $e \leq f$ and $\theta \subseteq \phi$.
Let $(t,x) \in Se\, \Box \, X \subseteq  Sf \, \Box \, X$.
Then $(e \cdot \phi)(t,x) = \phi (t,x) = \theta (t,x)$.
Thus $\theta = e \cdot \phi$.
\end{proof}

\begin{example}{\em  Let $M$ be any monoid with identity 1.
Then $M$ is a left restriction monoid if we define $m^{+} = 1$ for all $m \in M$.
In this case, the set of projections is just $\{1\}$.
Supported actions are simply monoid actions.
The product action of two actions has as underlying set the product of the two sets.
In this case, the exponential reduces to $Y^{X} = \mbox{hom}(M \times X,Y)$.}
\end{example}

In the light of Theorem~\ref{them:CCV1}, the following is a natural question.

\begin{question}
{\em Which Cartesian closed categories are equivalent to those of the form $S$-{\bf Sup} for some left restriction monoid $S$?
In the special case where $S$ is a Boolean left restriction monoid
and the actions are `factorizable'  this is answered in Theorem~\ref{them:marytwo}
and is equivalent to the main result of Garner \cite{G2024}.}
\end{question}

Let $[E|M]$ be a weak matched pair.
By an {\em $[E|M]$-set},  we mean a set $Y$ satisfying the following axioms: 
\begin{itemize}
\item[(MPA1).] The monoid $M$ acts on the set $Y$. We denote this action by $\cdot$.
\item[(MPA2).] For each $e \in E$, there is an equivalence relation $\equiv_{e}$ defined on $Y$. 
(There is a notation clash with the definition of a matched pair. Like Garner, we shall ignore it.)
\item[(MPA3).] $\equiv_{1}$ is the equality equivalence relation.
\item[(MPA4).] If $f \leq e$ in $E$ then $x \equiv_{e} y$ implies that $x \equiv_{f} y$.
\item[(MPA5).] If $m \equiv_{e} n$ in $M$ and $x \in Y$ then $m \cdot x \equiv_{e} n \cdot x$.
\item[(MPA6).] If $x \equiv_{e} y$ in $Y$ then for any $m \in M$ we have that $m \cdot x \equiv_{m \ast e} m \cdot y$.
\end{itemize}
If $Y$ and $Y'$ are $[E|M]$-sets, then a {\em homomorphism} from $Y$ to $Y'$ is a function $\alpha \colon Y \rightarrow Y'$
such that $\alpha$ is a map of $M$-actions and if $x \equiv_{e} y$ then $\alpha(x) \equiv_{e} \alpha (y)$.
We may therefore form the category of $[E|M]$-sets which we denote by $[E|M]$-{\bf Set}.

In the following result, observe that we need both the monoid and the supported action to be factorizable.

\begin{theorem}\label{them:mary} Let $S$ be a factorizable left restriction monoid.  
Then the category $S$-{\bf FSup} is equivalent to the category $[\mathsf{Proj}(S)|\mathsf{Tot}(S)]$-{\bf Set}.
\end{theorem}
\begin{proof} 
Let $(S,X,p)$ be a factorizable supported action.
Put $M = \mathsf{Tot}(S)$, a monoid, and $E = \mathsf{Proj}(S)$, a meet semilattice with top element.
We therefore have a weak matched pair $[E|M]$ using the definitions of  Lemma~\ref{lem:putty}.
If $X = \varnothing$ then  the empty set is an $[E|M]$-set on $\varnothing = Y$.
If $X \neq \varnothing$ then because our supported action is factorizable, the set $Y = X_{1}$ is non-empty.
The action of $S$ on $X$ induces an action of $M$ on $Y$, 
since if $s \in M$ and $x \in Y$ then $s \cdot x \in Y$ because the action is supported.
We shall denote this induced action also by $\cdot$.
Define the relation $\equiv_{e}$ on $Y$ for each $e \in E$ by $x \equiv_{e} y$, for $x,y \in Y$, iff $e \cdot x = e \cdot y$.
Thus both (MPA1) and (MPA2) hold.
(MPA3) also holds because $S$ is a monoid and acts on $X$, and so $1 \cdot x = x$ for all $x \in X$.
Suppose that $f \leq e$ and that $e \cdot x = e \cdot y$.
Then, $f \cdot (e \cdot x) = f \cdot (e \cdot y)$ and so $f \cdot x = f \cdot y$
by properties of actions.
Thus (MPA4) holds.
Suppose that $em = en$ then $(em) \cdot x = (en) \cdot x$.
It follows that $e \cdot (m \cdot x) = e \cdot (n \cdot x)$ by properties of actions.
Thus (MPA5) holds.
If $e \cdot x = e \cdot y$ then $(me) \cdot x = (me) \cdot y$. 
By axiom (LR6), $me = (me)^{+}m$.
Thus (MPA6) holds.
In all cases, we have proved that the set $Y$ is an $[E|M]$-set.

We now define a functor $A$ from $S$-{\bf FSup} to $[\mathsf{Proj}(S)|\mathsf{Tot}(S)]$-{\bf Set}.
We start with a supported action $(S,X,p)$.
As above, there is a $[\mathsf{Proj}(S)|\mathsf{Tot}(S)]$-set $Y$.
This defines $A$ on objects.
We now define $A$ on homomorphisms.
Let $(S,X',p')$ be a factorizable supported action and let $\theta \colon X \rightarrow X'$ be a homomorphism.
We shall need to deal with various cases.

(Case 1): Suppose that neither $X$ nor $X'$ is empty.
We are dealing with factorizable supported actions and so neither $X_{1}$ nor $X'_{1}$ is empty.
Put $\theta' = \theta | X_{1}$.
Suppose that $p(y) = 1$ then $p' (\theta (y)) = 1$.
It follows that $\theta' \colon X_{1} \rightarrow X'_{1}$ is a well-defined function which preserves the action.
Suppose that $x \equiv_{e} x'$ in $X_{1}$.
Then, by definition, $e \cdot x = e \cdot x'$.
It follows that $e \cdot \theta (x) = e \cdot \theta (x')$
since $\theta$ is a homomorphism.
Thus $\theta'$ is a homomorphism of $[\mathsf{Proj}(S)|\mathsf{Tot}(S)]$-sets.

(Case 2): Suppose that $X = \varnothing$.
In this case, the empty function is the only homomorphism from $(S,X,p)$ to $(S,X',p')$.
This induces the empty function between the corresponding $[E|M]$-sets.

(Case 3): Suppose that $X \neq \varnothing$ but $X' = \varnothing$.
In this case, there is no homomorphism from $(S,X,p)$ to $(S,X',p')$.

There is one point left to check that $A$ really is a functor: composition should be preserved.
Suppose that $(S,X,p)$, $(S,X',p')$ and $(S,X'',p'')$ are supported actions.
Again we have to examine cases: if $X$, $X'$ and $X''$ are all non-empty
and we have homomorphisms $\theta \colon X \rightarrow X'$ 
and $\phi \colon X' \rightarrow X''$ then $A(\phi)A(\theta) = A(\phi \theta)$.
In the other cases, either there is no homomorphism at all
or there is the empty homomorphism which composes with any map to give the empty homomorphism.
We have therefore proved that $A$ is a functor.

To show that these two categories are equivalent, we shall show that $A$ is full, faithful  and essentially surjective.

Let $(S,X,p)$ and $(S,Y,q)$ be factorizable supported actions where neither $X$ nor $Y$ is empty.
Let $\alpha \colon X_{1} \rightarrow Y_{1}$ be a homomorphism of $[\mathsf{Proj}(S)|\mathsf{Tot}(S)]$-sets.
Define $\theta \colon X \rightarrow Y$ by $\theta (x) = p(x) \cdot \alpha (x')$ where $x \leq x'$  and $x' \in X_{1}$,
where we have used the fact that our supported action is factorizable.
We first show that $\theta$ is well-defined.
Suppose that $x \leq x''$ where $x'' \in X_{1}$.
Then $x' \equiv_{p(x)} x''$ and so $\alpha (x') \equiv_{p(x)} \alpha (x'')$.
Thus $\theta$ is well-defined.
We now check that $\theta$ is a homomorphism. 

(HSA1) holds: Let $s \in S$.
Let $x\leq x'$ where $x' \in X_{1}$ and let $s \leq \hat{s}$, where $\hat{s} \in \mathsf{Tot}(S)$ since $S$ is factorizable.
Then $s \cdot x \leq \hat{s} \cdot x'$, by part (3) of Lemma~\ref{lem:daphne}.
Observe that $\hat{s} \cdot x' \in X_{1}$.
Thus, by definition, $\theta (s \cdot x) = p(s \cdot x) \cdot \alpha (\hat{s} \cdot x')$.
Now,  
$$p(s \cdot x) \cdot \alpha (\hat{s} \cdot x') = (sp(x))^{+} \cdot \alpha (\hat{s} \cdot x') = (sp(x))^{+}\hat{s} \cdot \alpha (x').$$
We have that $(sp(x))^{+} = (sp(x))^{+}s^{+}$.
Thus $(sp(x))^{+}\hat{s} = (sp(x))^{+}s = sp(x)$ using axiom (LR6).
We have proved that
$\theta (s \cdot x) = sp(x) \cdot \alpha (x') = s \cdot \theta (x)$. 

(HSA2) holds: Observe that $q\theta (x) = p(x)$ since if $x' in X_{1}$ then $\alpha (x') \in Y_{1}$
and so $q\alpha (x') = 1$.

In the case where at least one of the sets $X$ or $Y$ is empty,
then there is at most one function between the corresponding $[E|M]$-sets 
and so there is at most one function between the corresponding supported sets.
The key point is that the empty function lifts to the empty function. 
This proves that $A$ is full.

Faithfulness follows from the fact that in the case of factorizable supported actions on a non-empty set $X$,
homomorphisms with domain $X$ are determined by their values on $X_{1}$.
If $X$ is empty then we only have the empty function.

It remains to show that $A$ is essentially surjective.
To that end, 
let $Y$ be a $[\mathsf{Proj}(S)|\mathsf{Tot}(S)]$-set.
Define
$$X = \bigsqcup_{e \in \mathsf{Proj}(S)} Y/\equiv_{e}, \text{ a disjoint union}$$
and
$p \colon X \rightarrow \mathsf{Proj}(S)$
by $p([x]_{e}) = e$.
We now define what will turn out to be a supported action of $S$ on $X$.
Let $s = s^{+}\hat{s}$ where $\hat{s} \in \mathsf{Tot}(S)$.
Define 
$$s \bullet [x]_{e} = [\hat{s} \cdot x]_{(se)^{+}}.$$
We show that this is well-defined.
Let $[x]_{e} = [x']_{e}$ and 
$s = s^{+}\hat{s} = s^{+}\hat{t}$, where $\hat{s}, \hat{t} \in \mathsf{Tot}(S)$.
We prove that
$[\hat{s} \cdot x]_{(se)^{+}} = [\hat{t} \cdot x']_{(se)^{+}}$.
From $\hat{s} \equiv_{s^{+}} \hat{t}$ we get that $\hat{s} \cdot x \equiv_{s^{+}} \hat{t} \cdot x$
by axiom (MPA5).
From $x \equiv_{e} x'$ we get $\hat{t} \cdot x \equiv_{\hat{t} \ast e} \hat{t} \cdot x'$
by axiom (MPA6).
We now apply axiom (MPA4) twice.
From $(se)^{+} \leq s^{+}$,
we get $\hat{s} \cdot x \equiv_{(se)^{+}} \hat{t} \cdot x$,
and from $(se)^{+} \leq (\hat{t}e)^{+}$,
we get $\hat{t} \cdot x \equiv_{(se)^{+}} \hat{t} \cdot x'$.
It follows that	
$$\hat{s} \cdot x \equiv_{(\hat{s}e)^{+}} \hat{t} \cdot x'.$$
It is now easy to show that $(S,X,p)$ is a supported action.
Observe that this is actually factorizable: 
let $[x]_{e}$ be any element of $X$.
Then $e \bullet [x]_{1} = [x]_{e}$.
In this case,
the set $X_{1}$ consists of singleton sets $\{y\}$ (by axiom (MPA3)) where $y \in Y$.
There is an obvious isomorphism with the $Y$ we started with.

The $[\mathsf{Proj}(S)|\mathsf{Tot}(S)]$-set on the empty set arises from the only supported action of $S$ on the empty set.
\end{proof}

\begin{remark}\label{rem:agatha}
{\em With reference to Theorem~\ref{them:mary},
observe that
$$[x']_{f} \leq [x]_{e} \text{ iff } f \leq e \text{ and } x' \equiv_{f} x.$$ 
}
\end{remark}

\section{The Boolean case}

In this section, we shall refine what we have shown in the previous two sections
and make contact with Garner's work.
To this end, we shall not work with arbitrary left restriction monoids but those which are `Boolean'.
Our first goal, then, is to define what this means.

\begin{terminology}
{\em We shall assume the reader is familiar with Boolean algebras.
Furthermore,  we shall assume that $0 \neq 1$ (so, our Boolean algebras are always  `non-degenerate').
The complement of $e$ will be denoted by $\bar{e}$,
the meet of $e$ and $f$ by $ef$
and the join of $e$ and $f$ by $e \vee f$.}
\end{terminology}

Let $S$ be a left restriction monoid.
Define $a \sim_{+} b$, on $S$, and say that $a$ is {\em $+$-compatible} with $b$,
if $a^{+}b = b^{+}a$.
Being $+$-compatible is a necessary condition for a pair of elements to have a join.
If $S$ has a zero which is also a projection,
then we say that the elements $s$ and $t$ are {\em $+$-orthogonal}, denoted by $s \perp t$, if $s^{+}$ and $t^{+}$ are orthogonal (meaning $s^{+}t^{+} = 0$).
$+$-orthogonal elements are $+$-compatible.
The following result, which is well-known, shows that a pair of elements in a left restriction monoid which are bounded above are always $+$-compatible.

\begin{lemma}\label{lem:trump} Let $S$ be a left restriction monoid.
If $a,b \leq c$ then $a \sim_{+} b$.
\end{lemma}
\begin{proof} Suppose that $a,b \leq c$.
Then $a = a^{+}c$ and $b = b^{+}c$.
Thus $a^{+}b = a^{+}b^{+}c$ and $b^{+}a = b^{+}a^{+}c$.
Using the fact that the projections commute, 
we have shown that $a^{+}b = b^{+}a$.
\end{proof}

We say that a left restriction monoid $S$ is {\em Boolean} if it satisfies the following three axioms:
\begin{itemize}
\item[{\rm (B1).}] $\mathsf{Proj}(S)$ is a Boolean algebra with respect to the natural partial order
with the zero of the semigroup being the zero of the Boolean algebra, and the identity of the monoid being the
one of the Boolean algebra. 
The meet operation is simply multiplication.
\item[{\rm (B2).}] $+$-compatible pairs of elements have joins. We say that {\em binary joins exist}.
\item[{\rm (B3).}] Multiplication distributes over binary joins from the left and from the right;
we therefore have that 
$a(b \vee c) =  ab \vee ac$
and 
$(b \vee c)a = ba \vee ca$.
\end{itemize}

\begin{remark}{\em 
Despite appearances, our definition of Boolean left restriction monoid agrees with the definition
of Boolean restriction monoid given as \cite[Definition 3.2]{G2025}.}
\end{remark}

Homomorphisms of Boolean left restriction monoids are required to preserve binary joins.
The left-right dual of the following result is proved in \cite[Lemma 3.1]{L2025}.
It will be used many times.

\begin{lemma}\label{lem:induction} In any Boolean left restriction monoid,
if $a = b \vee c$ exists then $a^{+} = b^{+} \vee c^{+}$.
\end{lemma}

The following will not be needed later, but answers an obvious question.

\begin{lemma}\label{lem:people} Let $S$ be a Boolean left restriction monoid.
Let $s \in S$ and $e$ be any projection.
Then $s\,\bar{e} = \overline{(se)^{+}}\,s$.
\end{lemma}
\begin{proof} We have that $1 = e \vee \bar{e}$ and so $s = se \vee s \bar{e}$.
Similarly, $1 = (se)^{+} \vee \overline{(se)^{+}}$
and so $s = (se)^{+}s \vee \overline{(se)^{+}}\,s = se \vee \overline{(se)^{+}}\,s$.
We have therefore shown that $s\bar{e} \sim_{+} \overline{(se)^{+}}\,s$, by Lemma~\ref{lem:trump}.
The result will follow if we can show that $(s\bar{e})^{+} = (\overline{(se)^{+}}s)^{+}$.
This follows from standard Boolean algebra as we now show.
By Lemma~\ref{lem:induction},
$s^{+} = (se)^{+} \vee (s\bar{e})^{+}$
and
$s^{+} = (se)^{+} \vee (\overline{(se)^{+}}s)^{+}$.
But both of these are orthogonal joins.
Thus 
$(s\bar{e})^{+} = (\overline{(se)^{+}}s)^{+}$.
\end{proof}

Proposition~\ref{prop:map} below connects what we are doing here with Section~2 and Theorem~\ref{them:mary}.

\begin{proposition}\label{prop:map}
Every Boolean left restriction monoid is factorizable.
\end{proposition}
\begin{proof} Let $S$ be a Boolean left restriction monoid.
Let $s \in S$.
Then the elements $s$ and $\overline{s^{+}}$ are $+$-orthogonal.
If we put $\hat{s} = s \vee \overline{s^{+}}$ then we obtain a total element, by Lemma~\ref{lem:induction}, that is above $s$ in the natural partial order.
We may therefore write $s = s^{+}\hat{s}$.
It follows that Boolean left restriction monoids are automatically factorizable.
\end{proof}

\subsection{The construction of Boolean left restriction monoids from matched pairs}

We shall now strengthen our definition of weak matched pairs to the case where $E = B$ is actually a Boolean algebra.
This will then agree with the definition of matched pair that Garner gave.
Let $M$ be a monoid and $B$ a Boolean algebra.
We say that $[B|M]$ is a {\em matched pair} if the following axioms hold (due to Garner):
\begin{itemize}
\item[{\rm (MP1).}] The monoid $M$ acts on the set $B$.
\item[{\rm (MP2).}] $m \ast 1 = 1$ for all $m \in M$.
\item[{\rm (MP3).}] $m \ast (ef) = (m \ast e)(m \ast f)$.
\item[{\rm (MP4).}] $\equiv_{e}$ is a right congruence on $M$. The $\equiv_{e}$-class containing the element $m$ will be denoted by $[m]_{e}$.
\item[{\rm (MP5).}] $\equiv_{1}$ is the identity right congruence.
\item[{\rm (MP6).}] If $f \leq e$ then $m \equiv_{e} n$ implies that $m \equiv_{f} n$.
\item[{\rm (MP7).}] If $a \equiv_{e} a'$ then $ma \equiv_{m \ast e} ma'$ for any $m \in M$.
\item[{\rm (MP8).}] If $a \equiv_{e} a'$ then $e(a \ast f) = e(a' \ast f)$ for any $f \in B$.
\item[{\rm (MP9).}] $\equiv_{0}$ is the universal right congruence.
\item[{\rm (MP10).}] $m \ast (e \vee f) = m \ast e \vee m \ast f$.
\item[{\rm (MP11).}] If $m \equiv_{e} n$ and $m \equiv_{f} n$ then $m \equiv_{e \vee f} n$. 
\item[{\rm (MP12).}] Given any\ $m,n \in M$ and any $e \in B$, then there exists a $p \in M$ such that $p \equiv_{e} m$ and $p \equiv_{\bar{e}} n$. 
\end{itemize}
These are the same axioms as in Section~2 but augmented by the axioms (MP9)--(MP12) to take account of the extra structure in a Boolean algebra
compared with a semilattice.

Boolean left restriction monoids give rise to matched pairs, as first established by Garner.

\begin{lemma}\label{lem:anne} Let $S$ be a Boolean left restriction monoid.
Then $[\mathsf{Proj}(S)|\mathsf{Tot}(S)]$ is a matched pair.
\end{lemma}
\begin{proof} We know from Section~2, that the axioms (MP1)--(MP8) all hold.
It is immediate that axioms (MP9) and (MP10) hold.
To see that (MP11) holds, suppose that $m \equiv_{e} n$ and $m \equiv_{f} n$ both hold.
Then $em = en$ and $fm = fn$.
It is easy to check that $(e \vee f)m = (e \vee f)n$, and the result follows.
For axiom (MP12), put $p = em \vee \bar{e}n$.
This makes sense because $em$ and $\bar{e}n$ are $+$-orthogonal.
\end{proof}

The above result raises the question of whether every Boolean left restriction monoid
arises from such a matched pair. 
The answer is in the affirmative and is due to  Garner in \cite{G2025}.
We shall sketch a proof using Lemma~\ref{lem:amelia} below.

\begin{lemma}\label{lem:amelia} Let $[B|M]$ be a matched pair.
Suppose that $m \equiv_{ef} n$. Then there exists a $p \in M$ such that
$p \equiv_{e} m$ and $p \equiv_{f} n$.
Furthermore, if $p' \in M$ is any element such that 
$p' \equiv_{e} m$ and $p' \equiv_{f} n$ then
$p \equiv_{e \vee f} p'$.
\end{lemma}
\begin{proof} By axiom (MP12), there exists $p \in M$ such that $p \equiv_{e} m$ and $p \equiv_{\bar{e}} n$.
Observe that $f = ef \vee \bar{e}f$.
By axiom (MP6), we have that $p \equiv_{ef} m$ and $p \equiv{\bar{e}f} n$.
But $m \equiv_{ef} n$ so that $p \equiv_{ef} n$.
Thus $p \equiv_{ef} n$ and $p \equiv_{\bar{e}f} n$ and so by axiom (MP11),
we have that $p \equiv_{f} n$.
Suppose that $p' \equiv_{e} m$ and $p' \equiv_{f} n$.
Then $p' \equiv_{e} p$ and $p' \equiv_{f} p$.
It follows from axiom (MP11),  that $p' \equiv_{e \vee f} p$.
\end{proof}

We now have Theorem~\ref{them:martha} below.
This is proved in \cite{G2025} so we shall just sketch it here.
It can be considered a refinement of Theorem~\ref{them:outhwaite}, although it was actually proved first by Garner
and Theorem~\ref{them:outhwaite} is, in fact, its generalization.

\begin{theorem}\label{them:martha} 
Let $M$ be a monoid and $B$ a Boolean algebra. 
\begin{enumerate} 

\item Let $[B|M]$ be a matched pair. Then we can construct a Boolean left restriction monoid $S = S[B|M]$.
In addition, $[B|M] \cong [\mathsf{Proj}(S)|\mathsf{Tot}(S)]$.

\item Every Boolean left restriction monoid is isomorphic to a monoid of the form  $S[B|M]$.

\item The category of Boolean left restriction monoids is equivalent to the category of matched pairs. 

\end{enumerate}
\end{theorem}
\begin{proof} We have already proved that we can construct a left restriction monoid $S = S[B|M]$
in Lemma~\ref{lem:fiennes}.
It remains to show that $S$ is a Boolean left restriction monoid.
Suppose that $(e,[a]_{e}) \sim_{+} (f,[b]_{f})$.
Then $a \equiv_{ef} b$.
By Lemma~\ref{lem:amelia}, there is an element $p \in M$ such that $p \equiv_{e} a$ and $p \equiv_{f} b$.
Consider the element 
$$(e \vee f, [p]_{e \vee f}).$$
It is now easy to show that
$$(e,[a]_{e}) \vee (f,[b]_{f}) = (e \vee f, [p]_{e \vee f}).$$
The remainder of the proof follows from
\cite[Proposition 3.4, Theorem 3.5]{G2025}
and so is omitted.
\end{proof}

\subsection{Supported actions of  Boolean left restriction monoids and `actions' of matched pairs}

We need to modify what we mean by a supported action in the case where we are acting by Boolean left restriction monoids.

Let $S$ be a Boolean left restriction monoid.
If $X = \varnothing$ then all supported actions of $S$ on $X$ will be deemed Boolean.
If  $X \neq \varnothing$ then to have a supported action $(S,X,p)$ the following axioms {\em in addition to (E1) and (E2)} must be satisfied:
\begin{itemize}
\item[{\rm (E3).}] The poset $(X,\leq)$ has a minimum element $0$ such that $s \cdot 0 = 0$ for all $s \in S$.
\item[{\rm (E4).}] $0 \cdot x = 0$ for all $x \in X$. 
\item[{\rm (E5).}] If $x \approx y$ then $x \vee y$ exists in $X$. 
\item[{\rm (E6).}] If $x \approx y$ then $s \cdot (x \vee y) = s \cdot x \vee s \cdot y$.
\end{itemize}

We call such things  {\em Boolean supported actions};
this will always include the assumption that $S$ is a Boolean left restriction monoid.

\begin{remark}{\em Observe that axiom (E6) makes sense since if $x \approx y$ then for any $s$ we have that
$s \cdot x \approx s \cdot y$.}
\end{remark}

The following result was suggested by Francesco Tesolin.

\begin{lemma}\label{lem:childe} Let $(S,X,p)$ be a Boolean supported action on a non-empty set.
\begin{enumerate}
\item If $0$ is the smallest element of $X$ then $p(0) = 0$.
\item If $x \approx y$ then $p(x \vee y) = p(x) \vee p(y)$.
\item If $s \sim_{+} t$ then $s \cdot x \approx t \cdot x$ and $(s \vee t) \cdot x = s \cdot x \vee t \cdot x$.
\end{enumerate}
\end{lemma}
\begin{proof} (1)
By axiom (E3), we have that $0 \cdot 0 = 0$.
Thus $p(0) = p(0 \cdot 0) = (0p(0))^{+} = 0$, since $0^{+} = 0$.

(2) Suppose that $x \approx y$.
Since $x,y \leq x \vee y$ and $p$ is order preserving, we have that $p(x) \vee p(y) \leq p(x \vee y)$.
Suppose that $p(x), p(y) \leq e$, where $e$ is a projection.
Then $ep(x \vee y) = p(e \cdot (x \vee y))$ by axiom (E2).
By axiom (E6), we have that $p(e \cdot (x \vee y)) = p(e \cdot x \vee e \cdot y)$.
However, $p(x) \leq e$ implies that $e \cdot x = x$ by axiom (E1).
Similarly, $e \cdot y = y$.
We have proved that $p(x \vee y) \leq e$.
This proves the result.

(3) Let $s \sim_{+} t$.
By definition, this means that $t^{+}s = s^{+}t$.
Now $p(t \cdot x) \cdot (s \cdot x) = (tp(x))^{+}s \cdot x$.
But $(tp(x))^{+} \leq t^{+}$ and $(tp(x))^{+}t^{+}s \cdot x = (tp(x))^{+}s^{+}t \cdot x$.
We now use the fact that projections commute, so $(tp(x))^{+}s^{+} = s^{+}(tp(x))^{+}$.
By axiom (LR5), we have that 
$$s^{+}(tp(x))^{+} = (s^{+}tp(x))^{+} = (t^{+}sp(x))^{+} = t^{+}(sp(x))^{+}.$$
It follows that
$$p(t \cdot x) \cdot (s \cdot x) = p(s \cdot x) \cdot (t \cdot x).$$
We have shown that $s \cdot x \approx t \cdot x$.
Since $s \leq s \vee t$ we have that $s \cdot x \leq (s \vee t) \cdot x$.
Similarly, $t \cdot x \leq (s \vee t) \cdot x$.
It follows that $s \cdot x \vee t \cdot x \leq (s \vee t) \cdot x$.
By axiom (E2), $p((s \vee t) \cdot x) = ((s \vee t)p(x))^{+}$.
This is equal to $(sp(x))^{+} \vee (tp(x))^{+}$, where we have used Lemma~\ref{lem:induction}.
By (2) above, this is equal to $p(s \cdot x \vee t \cdot x)$.
The result now follows.
\end{proof}

Let $(S,X,p)$ and $(S,X',p')$ be Boolean supported actions.
In the case where $X$ is the empty set, then the unique function from $X$ to $X'$ will
be a {\em Boolean homomorphism}\footnote{This is, perhaps, not ideal terminology but is shorter than `a homomorphism of Boolean supported actions'. 
It specifically does not mean a `homomorphism of Boolean algebras'.}.
If $X$ is not empty but $X'$ is empty, then there is no homomorphism from $X$ to $X'$.
That leaves the case where both $X$ and $X'$ are non-empty.
In this case,  a {\em Boolean homomorphism} of such actions is a function $\theta \colon X \rightarrow X'$ 
which satisfies (HSA1) and (HSA2) and the further condition:
\begin{itemize}
\item[(HSA3):] If $x \vee y$ exists in $X$ then $\theta (x \vee y) = \theta (x) \vee \theta (y)$. 
\end{itemize}

The following is not needed but answers an obvious question.

\begin{lemma}\label{lem:nun} Let $(S,X,p)$ and $(S,X',p')$ be Boolean supported actions
where neither $X$ nor $X'$ is empty.
Let $\theta \colon X \rightarrow X'$ be a Boolean homomorphism.
Then the minimum element of $X$ is mapped to the minimum element of $X'$ by $\theta$.
\end{lemma}
\begin{proof} For the purpose of this proof, denote the minimum element of $X$ by $0_{X}$
and the minimum element of $X'$ by $0_{X'}$.
We have that $0_{X} = 0 \cdot 0_{X}$ by axiom (E4).
Thus 
$$\theta (0_{X}) = \theta (0 \cdot 0_{X}) = 0 \cdot \theta (0_{X}) = 0_{X'}.$$
\end{proof}

Let $S$ be a  Boolean left restriction monoid.
Then we can define the category, denoted by $S$-{\bf BSup}, whose objects are the Boolean supported actions $(S,X,p)$
and whose arrows are the Boolean homomorphisms.
The full subcategory of $S$-{\bf BSup} consisting of factorizable supported actions of $S$ will be denoted by $S$-{\bf FBSup}.

The next result is the analogue of Proposition~\ref{prop:map} and completes Lemma~\ref{lem:empty} in the Boolean case.

\begin{proposition}\label{prop:lepen} Let $S$ be a Boolean left restriction monoid.
Then the supported action $(S,X,p)$, where $X$ is non-empty, is factorizable iff $p$ is surjective.
\end{proposition}
\begin{proof} Suppose first that $(S,X,p)$ is factorizable.
Let $e$ be any projection and let $x$ be any element such that $p(x) = 1$ (that such an element $x$ exists is part of the definition
of factorizability when $X$ is non-empty).
Define $y = e \cdot x$.
Then $p(y) = p(e \cdot x) = ep(x) = e$.
Thus $p$ is surjective.
Conversely, suppose that $p$ is surjective.
We prove that $X$ is factorizable.
Let $x \in X$.
Put $e = p(x)$.
Because $p$ is surjective there exists $x' \in X$ such that $p(x') = \bar{e}$.
Observe that $x \approx x'$ and so $\hat{x} = x \vee x'$ exists by (E5).
By part (2) of Lemma~\ref{lem:childe}, $p(\hat{x}) = 1$.
Clearly, $x \leq \hat{x}$, and so $X$ is factorizable.
 \end{proof}

We now refer back to Examples~\ref{exs:examples}.
The proof of the following is straightforward.
We use the notation of part (2) of Examples~\ref{exs:examples}.

\begin{lemma}\label{lem:bird} Let $S$ be a Boolean left restriction monoid 
and let $e$ be any projection in $S$.
Then $(S,Se,p_{e})$ is a Boolean supported action.
\end{lemma}

\begin{lemma}\label{lem:newdegaulle}
Let $S$ be a Boolean left restriction monoid.
The category $S$-{\bf BSup} has a unique initial object.
\end{lemma}
\begin{proof}
The initial object of the category is $(S,\varnothing,p)$.
Observe that it is both Boolean and factorizable.
\end{proof}

With reference to Proposition~\ref{prop:sophie}, we have the following.

\begin{lemma}\label{lem:gopdem} Let $S$ be a Boolean left restriction monoid.
Let $(S,X,p)$ and $(S,Y,q)$ be Boolean supported actions 
\begin{enumerate}
\item $X \,\Box \, Y$ is a Boolean supported action.
\item The maps $\pi_{1}$ and $\pi_{2}$ are both Boolean homomorphisms.
\item If both $(S,X,p)$ and $(S,Y,q)$ are factorizable then 
$X \,\Box \, Y$ is factorizable.
\end{enumerate}
\end{lemma}
\begin{proof} (1) If $X \,\Box \, Y$ is empty then it is Boolean.
Suppose that $X \,\Box \, Y$ is not empty.
Let $0_{X}$ be the minimum element of $X$ and let
$0_{Y}$ be the minimum element of $Y$.
By part (1) of Lemma~\ref{lem:childe}, we have that $p(0_{X}) = 0 = q(0_{Y})$.
Thus $(0_{X}, 0_{Y}) \in X \,\Box \, Y$.
It is now easy to show that (E3) and (E4) hold. 
Suppose that $(x,y) \approx (x',y')$.
Then $x \approx x'$ and $y \approx y'$.
Then $(x \vee x', y \vee y') \in X \,\Box \, Y$
using the properties of $p$ and $q$.
It is easy to show that $(x,y) \vee (x',y') = (x \vee x', y \vee y')$
and that (E5) and (E6) hold.

(2) The proof is immediate.

(3) If $X \,\Box \, Y$ is empty then it is factorizable.
Suppose that $X \,\Box \, Y$ is not empty.
Let $(x,y) \in X\, \Box \,Y$.
By assumption, $x \leq x'$ where $p(x') = 1$, for some $x'$, and $y \leq y'$ where $q(y') = 1$, for some $y'$.
We then have that $(x,y) \leq (x',y')$ where $r(x',y') = 1$.
\end{proof}

We now prove the analogue of Proposition~\ref{prop:sophie}.

\begin{proposition}\label{prop:newsophie}
Let $S$ be a Boolean left restriction monoid.
Then the category $S$-{\bf BSup} (resp.   $S$-{\bf FBSup})  has 
a terminal object, and binary products.
\end{proposition}
\begin{proof} We work in the category $S$-{\bf BSup} first.

The terminal object is the action of  part (1) of Examples~\ref{exs:examples} since $\mathsf{Proj}(S)$  is a Boolean algebra.

Suppose that $(S,X,p)$ and $(S,Y,q)$ are Boolean supported actions.
If $X \, \Box \, Y$ is not the initial object then we proved in Lemma~\ref{lem:gopdem} that it is a Boolean supported action.

If either $X$ or $Y$ is initial  so too is $X \, \Box \, Y$ and so belongs to the category.
To prove that $X \, \Box \, Y$ has the correct universal property,
we just need to check that the map $\gamma$, defined in the proof of Proposition~\ref{prop:sophie},
is a Boolean homomorphism.
It is, because it is just the product of two maps which are themselves Boolean homomorphisms.

Everything we have said above goes through in the factorizable case.
\end{proof}

The following result will play an essential role later.

\begin{lemma}\label{lem:laura} Let $S$ be a Boolean left restriction monoid.
Suppose that neither $(S,X,p)$ nor $(S,Y,q)$ are initial factorizable Boolean supported actions.
Then there is a Boolean homomorphism from $Se \, \Box \, X$ to $Y$ for any projection $e$.
\end{lemma}
\begin{proof} Let $e$ be any projection in $S$.
Then, since $Y$ is not initial, $q$ is a surjection.
Choose any $y \in Y$ such that $q(y) = e$.
Define the function $\phi \colon Se \, \Box \, X \rightarrow Y$ by $\phi (a,x) = a \cdot y$.
(HSA1) holds: Let $s \in S$.
Then $\phi (sa, s \cdot x) = (sa) \cdot y$
whereas $s \cdot \phi (a,x) = s \cdot (a \cdot y)$.
(HSA2) holds: We calculate.
$$q\phi (a,x) = q(a \cdot y) = (aq(y))^{+} = (ae)^{+} = a^{+} = r(a,x).$$ 
(HSA3) holds: suppose that $(a,x) \approx (a',x')$.
Then $(a,x) \vee (a',x') = (a \vee a', x \vee x')$.
By definition, 
$$\phi (a \vee a', x \vee x') = (a \vee a') \cdot y = a \cdot y \vee a' \cdot y = \phi (a,x) \vee \phi (a',x').$$
\end{proof}

The next result tells about the  what the order is and how to compute joins in $X \, \Box \, Y$.

\begin{lemma}\label{lem:saul} Let $S$ be a Boolean left restriction monoid.
Let $(S,X,p)$ and $(S,Y,q)$ be Boolean supported actions and suppose that
$X \, \Box \, Y$ is not initial.
\begin{enumerate}

\item In $X \, \Box \, Y$, we have that $(x,y) \leq (u,v)$ iff
$x \leq u$ and $y \leq v$.

\item If $(x,y) \approx (x',y')$ in $X \, \Box \, Y$,
then $(x,y) \vee (x',y') = (x \vee x', y \vee y')$.

\end{enumerate}
\end{lemma}
\begin{proof} (1) Suppose that $(x,y) \leq (u,v)$.
Then, by definition, $(x,y) = r(x,y) \cdot (u,v)$.
But $r(x,y) = p(x) = q(y)$.
Thus $x \leq u$ and $y \leq v$.
The converse is proved by reversing the steps above.

(2) By Lemma~\ref{lem:gop}, if $(x,y) \approx (x',y')$ 
then $x \approx x'$ in $X$ and $y \approx y'$ in $Y$.
Thus $x \vee x'$ exists in $X$ and $y \vee y'$ exists in $Y$.
By part (2) of Lemma~\ref{lem:childe}, we have that
$$p(x \vee x') = p(x) \vee p(x') = q(y) \vee q(y') = q(y \vee y').$$
Thus $(x \vee x', y \vee y')$ is an element of $X \, \Box \, Y$.
We now prove that it is the appropriate join.
We calculate first $r(x,y) \cdot (x \vee x', y \vee y')$.
If we use the fact that $x \approx x'$ and $y \approx y'$,
then we get $(x,y)$. Thus $(x,y) \leq (x \vee x', y \vee y')$.
By a similar argument, $(x',y') \leq (x \vee x', y \vee y')$.
Suppose that $(u,v)$ is an element of $X \, \Box \, Y$
such that $(x,y), (x',y') \leq (u,v)$.
From the definitions, we quickly find that $x,x' \leq u$ and $y,y' \leq v$.
Thus $x \vee x' \leq u$ and $y \vee y' \leq v$.
It is now easy to show that $(x \vee x', y \vee y') \leq (u,v)$.
\end{proof}

\begin{theorem}\label{them:CCV2} Let $S$ be a Boolean left restriction monoid.
Then the category  $\mathscr{C} = S$-{\bf FBSup} is Cartesian closed.
\end{theorem}
\begin{proof} This can be regarded as a refinement of Theorem~\ref{them:CCV1} to Boolean left restriction monoids.\\

 \underline{We shall deal first with the case where all objects are non-initial.}\\

Let $(S,X,p)$ and $(S,Y,q)$ be non-initial factorizable Boolean supported actions.
We proved in Lemma~\ref{lem:bird} that for any projection $e$ we have that $(S,Se,p_{e})$  is a Boolean supported action.
Thus, by Proposition~\ref{prop:newsophie}, $Se \, \Box \, X$ is a Boolean supported action.
The set $\mbox{hom}(Se \, \Box \, X, Y)$ is now the set of all {\em Boolean} homomorphisms.
Define
$$Y^{X} = \bigcup_{e \in \mathsf{Proj}(S)} \{e\} \times  \mbox{hom}(Se \, \Box \, X, Y)
\text{ and } 
\pi \colon Y^{X} \rightarrow \mathsf{Proj}(S)
\text{ by } \pi ((e,\theta)) = e.$$
Observe that we have defined $Y^{X}$ to be a disjoint union.\\
\noindent
{\bf Notation.} To ease notation, we shall usually write elements of $Y^{X}$ simply as $\theta$ when we know that
$\theta \in \mbox{hom}(Se \, \Box \, X, Y)$.\\

We need to prove first that $(S,Y^{X}, \pi)$ is a Boolean supported action.\\

\noindent
{\em If $\theta \in Y^{X}$ then $a \cdot \theta \in Y^{X}$: }
We know that $a \cdot \theta$ is a homomorphism of merely supported actions.
We now have to prove the stronger result that it is a Boolean homomorphism.
Let $\theta \in \mbox(Se \, \Box \, X, Y)$ be a Boolean homomorphism.
We prove that $a \cdot \theta$ is a Boolean homomorphism.
Suppose that $(t,x), (t',x') \in  S(ae)^{+} \, \Box \, X$ and $(t,x) \approx (t',x')$.
Then $t \approx t'$ and $x \approx x'$ by Lemma~\ref{lem:gop}; we shall use this result many times below.
It follows by Lemma~\ref{lem:saul}, that $(t,x) \vee (t',x') = (t \vee t', x \vee x')$.
Observe that 
$$ (t \vee t', x \vee x') \in S(ae)^{+} \, \Box \, X.$$
Thus 
$$(a \cdot \theta)(t \vee t', x \vee x') = (ta \vee t'a, (ta)^{+} \cdot x \vee (t'a)^{+} \cdot x')$$
where we have used the fact that $t' \approx t$.
It follows that
$$(a \cdot \theta)(t \vee t', x \vee x') = (a \cdot \theta)(t,x) \vee (a \cdot \theta)(t',x').$$
We have proved that $a \cdot \theta$ is a Boolean homomorphism.
We therefore have a well-defined supported action of $S$ on the set $Y^{X}$.

We now show that we have a Boolean supported action.

{\em (E3) and (E4) hold: } 
The only element in $\{0\} \,\Box \, X$ is $\{ (0,0) \}$.
Thus the zero element in $Y^{X}$ is 
$\zeta \colon \{(0,0)\} \rightarrow \{ 0\}$.
That this really is the minimum element follows from Lemma~\ref{lem:saul}.
This element is fixed under the action by elements of $S$.\\

{\em (E5) holds:} Let  $\pi (\theta) = e$ and $\pi (\phi) = f$
and that  $\theta \approx \phi$.
This means that $e \cdot \phi = f \cdot \theta$.
We prove that  $\theta \vee \phi$ exists and describe it.
Let $(t,x) \in S(e \vee f)\, \Box \, X$.
Define
$$(\theta \vee \phi)(t,x) = \theta (te, (te)^{+} \cdot x) \vee \phi (tf, (tf)^{+} \cdot x).$$
It is important to note that $\theta \vee \phi$ is just notation at the moment.
We have to show first that the right hand side is defined in $Y$.
Thus, we need to show that in $Y$, we have that $\theta (te, (te)^{+} \cdot x) \approx  \phi (tf, (tf)^{+} \cdot x)$
But $q\theta (te, (te)^{+} \cdot x) = (te)^{+}$ and $q\phi (tf, (tf)^{+} \cdot x) = (tf)^{+}$.
By axiom (LR6), we can show that $(te)^{+} \cdot  \theta (te, (te)^{+} \cdot x) = (tf)^{+} \cdot \phi (tf, (tf)^{+} \cdot x)$
using the fact that $\theta$ and $\phi$ are Boolean homomorphisms.
Thus the right hand side is defined.

We need to check that $\theta \vee \phi$ is a Boolean homomorphism.
(HSA1) holds: Because it holds for $\theta$ and $\phi$.
(HSA2) holds: From the above, we have that
$\theta (te, (te)^{+} \cdot x) \approx  \phi (tf, (tf)^{+} \cdot x)$,
which, together with the fact that $t = t(e \vee f) = te \vee tf$,
proves the result.
(HSA3) holds: Suppose that $(t,x) \approx (t',x')$ in $S(e \vee f)\, \Box \, X$.
Then $t \approx t'$ and $x \approx x'$.
Then $(t,x) \vee (t',x') = (t \vee t', x \vee x')$.
We have to prove that $(\theta \vee \phi) (t \vee t', x \vee x') = (\theta \vee \phi)(t,x) \vee (\theta \vee \phi)(t',x')$, 
but this this uses the fact that $x \approx x'$ and some routine calculations.
Observe that, by definition, $\pi (\theta \vee \phi) = e \vee f$.

We shall now show that our notation is justified.
We shall use the description of the order in $Y^{X}$ we gave in Lemma~\ref{lem:kate}.
First, we prove that $\theta \leq \theta \vee \phi$.
Suppose that  $(t,x) \in Se\, \Box \, X$.
Define
$$(\theta \vee \phi)(t,x) = \theta (te, (te)^{+} \cdot x) \vee \phi (tf, (tf)^{+} \cdot x) = \theta (t, x) \vee \phi (tfe, (tef)^{+} \cdot x).$$
But $\phi (tfe, (tef)^{+} \cdot x) = (e \cdot \phi)(tf, (tf)^{+} \cdot x)$ since $(tef)^{+} \leq (tf)^{+}$.
However, we are given that $f \cdot \theta = e \cdot \phi$.
Thus $\phi (tfe, (tef)^{+} \cdot x) = (f \cdot \theta)(tf, (tf)^{+} \cdot x) = \theta (tf, (tf)^{+} \cdot x)$.
But $(tf, (tf)^{+} \cdot x) \leq (t,x)$ and so $\theta (tf, (tf)^{+} \cdot x) \leq \theta (t, x)$.
We have therefore proved that $\theta \leq \theta \vee \phi$.
We may similarly show that $\phi \leq \theta \vee \phi$.

Suppose that $\theta, \phi \leq \psi$, where $\psi \in Y^{X}$.
Observe that $(e \vee f) \leq \pi (\psi)$.
We have to prove that $\theta \vee \phi \leq \psi$.
By assumption, $\theta = e \cdot \psi$ and $\phi = f \cdot \psi$.
We need to prove that $\theta \vee \phi = (e \vee f) \cdot \psi$.
Let $(t,x) \in S(e \vee f) \, \Box \, X$.
Then  $((e \vee f) \cdot \psi)(t,x) = \theta (t,x) \vee \phi (t,x) = (\theta \vee \phi)(t,x)$
by direct computation.
We have thus shown that $\theta \vee \phi$ is given by the above expression.

{\em  (E6) holds:} 
Let $\pi (\theta) = f$ and $\pi (\phi) = e$,
$a \in S$,
and suppose that $\theta \approx \phi$.
Then 
$$(\theta \vee \phi)(t,x) = \theta (tf, (tf)^{+} \cdot x) \vee \phi (te, (te)^{+} \cdot x).$$
Calculate $(a \cdot (\theta \vee \phi))(t,x)$.
This is equal to 
$$(a \cdot \theta)(t,x) \vee (a \cdot \phi)(t,x)$$
which proves that $a \cdot (\theta \vee \phi) = a \cdot \theta \vee a \cdot \phi$.\\

We therefore have a Boolean supported action $(S,Y^{X}, \pi)$.\\

{\em The map $\hat{g}$ is a Boolean homomorphism:}
(HSA3) holds: Suppose that $z,z' \in Z$ and $z \approx z'$.
We need to prove that $\hat{g}(z \vee z') = \hat{g}(z) \vee \hat{g}(z')$.
By definition, $\pi (\hat{g}(z)) = r(z)$ and $\pi (\hat{g}(z')) = r(z')$.
Thus, let $(t,x) \in S(r(z) \vee r(z')) \, \Box \, X$.
By definition
$$\hat{g}(z \vee z')(t,x) = g(t \cdot (z \vee z'), (tr(z \vee z'))^{+} \cdot x).$$
On the other hand, a small calculation from the definitions shows that
$$(\hat{g}(z) \vee \hat{g}(z')) (t,x) = g(t \cdot z, (tr(z))^{+} \cdot x) \vee g(t \cdot z', (tr(z'))^{+} \cdot x).$$
The result now follows from the fact that
$t \cdot (z \vee z') = t \cdot z \vee t \cdot z'$,
$(tr(z \vee z'))^{+} \cdot x = (tr(z))^{+} \cdot x \vee (tr(z'))^{+} \cdot x$,
and the fact that $g$ is a Boolean homomorphism.\\

{\em The map $\mathbf{eval}$ is a Boolean homomorphism: }
(HSA3) holds: Let $\theta, \phi \in Y^{X}$
where $\pi (\theta) = e$ and $\pi (\phi) = f$.
We suppose that $e = p(x)$ and $f = p(y)$
and that $(\theta,x) \approx (\phi,y)$.
Then $\theta \approx \phi$ and $x \approx y$.
It follows that $f \cdot x = e \cdot y$.
Then $(\theta, x) \vee (\phi,y) = (\theta \vee \phi, x \vee y)$.
By definition $\mathbf{eval}(\theta \vee \phi, x \vee y) = (\theta \vee \phi)(e \vee f, x \vee y)$.
But 
$$(\theta \vee \phi)(e \vee f, x \vee y) = \theta ((e \vee f)e, e \cdot (x \vee y)) \vee \phi ((e \vee f)f, f \cdot (x \vee y)).$$
We now use the fact that $ef \leq e,f$ and that $x \approx y$ and so
$$\theta ((e \vee f)e, e \cdot (x \vee y)) \vee \phi ((e \vee f)f, f \cdot (x \vee y)) = \theta (e,x) \vee \phi (f,y).$$
We have therefore proved that 
$$\mathbf{eval}(\theta \vee \phi, x \vee y) = \mathbf{eval}(\theta, x) \vee \mathbf{eval}(\phi, y).$$

{\em $Y^{X}$ is factorizable: }
Let $\theta$ be any element of $Y^{X}$,
where
$\theta \colon Se \, \Box \, X \rightarrow Y$.
We need to prove that there is an element $\phi \in Y^{X}$
such that $\theta \leq \phi$ and $\pi (\phi) = 1$.
This we shall now do.
The first step is to find our candidate for $\phi$.
By Proposition~\ref{prop:lepen}, the map $q$ is surjective.
This means there is a $y \in Y$ such that $q(y) = \bar{e}$. 
Define the function $\phi \colon S  \, \Box \, X \rightarrow Y$
by 
$$\phi (a,x) = \theta (ae, (ae)^{+} \cdot x) \vee a \cdot y.$$
We show that this definition makes sense.
Both $\theta (ae, (ae)^{+} \cdot x)$ and  $a \cdot y$ belong to $Y$ so
we need to show that $\theta (ae, (ae)^{+} \cdot x) \approx  a \cdot y$
in order to prove that their join is defined.
Direct calculation shows that $q\theta (ae, (ae)^{+} \cdot x) = (ae)^{+}$
and  $q(a \cdot y) = (a\bar{e})^{+}$.
Now, 
$$(a\bar{e})^{+} \cdot \theta (ae, (ae)^{+} \cdot x) = \theta ((a\bar{e})^{+}ae, (a\bar{e})^{+}(ae)^{+} \cdot x).$$
But by axiom (LR6), we have that $(a\bar{e})^{+}a = a\bar{e}$.
Thus $(a\bar{e})^{+}ae = 0$.
It follows that  $\theta ((a\bar{e})^{+}ae, (a\bar{e})^{+}(ae)^{+} \cdot x) = \theta (0,0)$.
Similarly, $(ae)^{+}a \cdot y = 0$ by axiom (E4).
This proves that the join is defined, and $\phi$ defines a function from $S\, \Box \, X$ to $Y$.

We now need to prove that the $\phi$ we have defined is a Boolean homomorphism.

(HSA1) holds: Let $s \in S$.
By definition, $\phi (sa, s \cdot x) = \theta (sae, (sae)^{+}s\cdot x) \vee sa \cdot y$.
By axiom (LR6), we have that $(sae)^{+}s = s(ae)^{+}$.
Thus $\phi (sa, s \cdot x)= \theta (sae, s(ae)^{+} \cdot x) \vee sa \cdot y$.
We can now take the $s$ out using the fact that multiplication distributes over joins from the left and the fact that $\theta$ is a homomorphism,
and get the desired result.

(HSA2) holds: If we calculate $q\phi (a,x)$, we get $(ae)^{+} \vee (a\bar{e})^{+} = a^{+}$ by Lemma~\ref{lem:induction}.
But this is just $r(a,x)$.

(HSA3) holds: Let $(a,x),(b,x') \in S  \, \Box \, X$ such that $(a,x) \approx (b,x')$.
Thus the element  $(a \vee b, x \vee x')$ exists and belongs to  $S  \, \Box \, X$. 
By definition, 
$$\phi (a \vee b, x \vee x') = \theta ((a \vee b)e, ((a \vee b)e)^{+} \cdot (x \vee x')) \vee (a \vee b) \cdot y.$$
But
$$\theta ((a \vee b)e, ((a \vee b)e)^{+} \cdot (x \vee x'))
=
\theta (ae \vee be, (ae)^{+} \cdot x \vee (ae)^{+} \cdot x' \vee (be)^{+} \cdot x \vee (be)^{+} \cdot x').$$
However, using the fact that $x \approx x'$, 
we have that 
$(ae)^{+} \cdot x' = (ae)^{+}b^{+} \cdot x$ and $(be)^{+} \cdot x = (be)^{+}a^{+} \cdot x'$
In addition, $(ae)^{+}b^{+} \leq (ae)^{+}$ and $(be)^{+}a^{+} \leq (be)^{+}$.
Thus 
$$\phi (a \vee b, x \vee x') 
= \theta (ae \vee be, (ae)^{+} \cdot x \vee be)^{+} \cdot x') \vee (a \cdot y) \vee (b \cdot y)$$ 
which in turn is equal to
$$ \theta (ae, (ae)^{+} \cdot x) \vee \theta (be, (be)^{+} \cdot x') \vee (a \cdot y) \vee (b \cdot y)$$
using the fact that $\theta$ is a Boolean homomorphism.
But this is equal to $\phi (a,x) \vee \phi (b,x')$.

We have shown that  $\phi$ really is an element of $Y^{X}$, and, in addition, $\pi (\phi) = 1$.
By definition, $(e \cdot \phi)(a,x) = \phi (ae, (ae)^{+} \cdot x)$.
But this equals $\theta (ae, (ae)^{+} \cdot x)$.
If we now select $(a,x) \in Se \, \Box \, X$ then, in this case,
$e \cdot \phi = \theta$.
We have proved that $\theta \leq \phi$.
This proves that $Y^{X}$ is factorizable.\\

\underline{We now deal with the case where some of the objects involved are initial.}\\

We shall use throughout the fact  that the initial object of our category is strict. 
We deal first with the very definition of $Y^{X}$.
There are four cases we have to contend with.
\begin{enumerate}

\item {\em $X$ is not initial and $Y$ is not initial.}
By Lemma~\ref{lem:laura}, the set $\mbox{hom}(Se \, \Box \, X,Y)$ is non-empty for any projection $e$.
Thus $Y^{X}$ is not initial and has a surjective support and so is factorizable.

\item {\em $X$ is initial and $Y$ is not initial.}
Then $Se \, \Box \, X$ is initial.
In this case, the set $Y^{X}$ reduces to the set $\{(e,\zeta_{Y}) \colon e \in \mathsf{Proj}(S)\}$.
The function $\pi$ does the following: $\pi (e,\zeta_{Y}) = e$.
Observe that $a \cdot (e,\zeta_{Y}) = ((ae)^{+},\zeta_{Y})$.
It follows that in this case, $(S,Y^{X}, \pi)$ is isomorphic (as expected) to the terminal object in the category which is a
factorizable Boolean supported action.

\item {\em $X$ is not initial and $Y$ is initial.}
Since $X$ is not initial its support is surjective.
Thus $Se \, \Box X$ is not initial.
It follows that $\mbox{hom}(Se \, \Box X,Y)$ is the empty set for every projection $e$.
Thus $Y^{X}$ is initial.

\item {\em $X$ is initial  and $Y$ is initial.} In this case, $Y^{X}$ is isomorphic to the terminal object (as expected).
\end{enumerate}

We have proved that $Y^{X}$ is an object of the category.
We now have to prove that this object has the requisite universal properties in these four cases.

\begin{enumerate}

\item {\em $X$ is not initial  and $Y$ is not initial.}
In this case, $Y^{X}$ is not initial and factorizable and so has a surjective support.
By assumption, $X$ is not initial and so has a surjective support.
It follows that $Y^{X} \, \Box \, X$ is not initial.
If $Z \, \Box \, X$ is not initial then we are back to the situation we dealt with in the first part of the proof.
We therefore have only to deal with the situation where $Z \, \Box \, X$ is initial.
In this case, the maps with domain $Z \, \Box \, X$ are uniquely determined.

\item {\em $X$ is initial and $Y$ is not initial.}
If $X$ is initial then $Y^{X} \, \Box \, X$ is initial.
So, $Z \, \Box \, X$ is initial. 
It follows that $Z \, \Box \, X = Y^{X} \, \Box \, X$.

\item {\em $X$ is not initial and $Y$ is initial.}
In this case, $Y^{X}$ is initial.
Thus $Y^{X} \, \Box \, X$ is initial.
There will only be a homomorphism from 
$Z \, \Box \, X$  to $Y$ if  $Z \, \Box \, X$ is initial,
in which case, $Z \, \Box \, X = Y^{X} \, \Box \, X$.

\item {\em $X$ is initial and $Y$ is initial.}
In this case, $Z \, \Box \, X = Y^{X} \, \Box \, X =Y$. 
\end{enumerate}

In all the cases above, the axioms for a Cartesian closed category hold.

\end{proof}

Let $B$ be a Boolean algebra and $M$ a monoid such that $[B|M$] is a matched pair.
Let $X$ be any set.
We say that $X$ is a {\em $[B|M]$-set} if the following axioms hold:
\begin{itemize}
\item[{\rm (MPA1).}] The monoid $M$ acts on the set $X$.
\item[{\rm (MPA2).}] For each $e \in B$ there is an equivalence relation $\equiv_{e}$ on $X$.
\item[{\rm (MPA3).}] $\equiv_{1}$ is the identity equivalence.
\item[{\rm (MPA4).}] If $f \leq e$ then $x \equiv_{e} y$ implies that $x \equiv_{f} y$.
\item[{\rm (MPA5).}] If $m \equiv_{e} n$ in $M$ and $x \in X$ then $m \cdot x \equiv_{e} n \cdot x$.
\item[{\rm (MPA6).}] If $x \equiv_{e} y$ in $X$ then for any $m \in M$ we have that $m \cdot x \equiv_{m \ast e} m \cdot y$.
\item[{\rm (MPA7).}] $\equiv_{0}$ is the universal equivalence.
\item[{\rm (MPA8).}] If $x \equiv_{e} x'$ and $x \equiv_{f} x'$ then $x \equiv_{e \vee f} x'$. 
\item[{\rm (MPA9).}] Given any $x,y \in X$ and any $e \in B$, then there exists a $z \in X$ such that $z \equiv_{e} x$ and $z \equiv_{\bar{e}} y$. 
\end{itemize}
These are the same axioms as in Section~3, but augmented by the axioms (MPA7), (MPA8) and (MPA9) to take account of the extra structure in a Boolean algebra.
{\em Homomorphisms} of such $[B|M]$-sets are defined as before.

{\em We now come to the main theorem of our paper.}
It is a  refinement of Theorem~\ref{them:mary} in the case of  Boolean left restriction monoids.
Recall that by Proposition~\ref{prop:map}, we have proved that every Boolean left restriction monoid is automatically factorizable.

\begin{theorem}\label{them:marytwo} Let $S$ be a Boolean left restriction monoid.  
The category $S$-{\bf FBSup} is equivalent to the category $[\mathsf{Proj}(S)|\mathsf{Tot}(S)]$-{\bf Set}.
\end{theorem}
\begin{proof} 
Put $M = \mathsf{Tot}(S)$, $B = \mathsf{Proj}(S)$. 
In fact, $[B|M]$ is a matched pair by Lemma~\ref{lem:anne}. 
We define a functor $A$ from $S$-{\bf FBSup} to $[B|M]$-{\bf Set}.
Let $(S,X,p)$ be a factorizable Boolean supported action.
If $X$ is empty then put  $X_{1} = \varnothing$, which is a $[B|M]$-set.
If $X$ is non-empty then $p$ is surjective.
Put $X_{1} = p^{-1}(1)$.
Then $X_{1}$ is a $[B|M]$-set by Theorem~\ref{them:mary}.
It just remains to check that the additional axioms hold.
It is straightforward to show that (MPA7) and (MPA8) hold.
To show that axiom (MPA9) holds,
let $e \in B$ and $x,y \in X_{1}$.
Recall that $x \equiv_{e} y$ iff $e \cdot x = e \cdot y$.
Observe that
$e \cdot x \approx \bar{e} \cdot y$
because $e$ is orthogonal to $\bar{e}$.
Thus $z = e \cdot x \vee \bar{e} \cdot y$ is well-defined.
But $z \equiv_{e} x$ and $z \equiv_{\bar{e}} y$.
Thus (MPA9) really holds.
Let $(S,X,p)$ and $(S,X',p')$ be supported actions
and
let $\theta \colon X \rightarrow X'$ be a homomorphism of supported actions.
The proof that $\theta | X_{1} \colon X_{1} \rightarrow X'_{1}$ defines a homomorphism of $[B|M]$-sets
is now routine.
It is now easy to see that we have defined a functor from the category of factorizable Boolean supported actions of $S$
to the category {\bf $[B|M]$-sets}.

We now prove that $A$ is full, faithful and essentially surjective.
The fact that $A$ is faithful is inherited from Theorem~\ref{them:mary}.
We show that $A$ is full.
Let $(S,X,p)$ and $(S,Y,q)$ be factorizable supported actions of the Boolean left restriction monoid $S$.
Let $\alpha \colon X_{1} \rightarrow Y_{1}$ be a homomorphism of $[B|M]$-sets.
We may define a homomorphism of supported actions $\theta \colon X \rightarrow Y$
by defining $\theta (x) = p(x) \cdot \alpha (x')$ where $x \leq x'$ and $x' \in X_{1}$. 
We need to check that $\theta$ preserves binary joins.
Let $x,y \in X$ such that $x \approx y$.
Then $p(y) \cdot x = p(x) \cdot y$.
We need to prove that $\theta (x \vee y) = \theta (x) \vee \theta (y)$.
Let $z' \in X_{1}$ be such that $x \vee y \leq z'$.
Then $x = p(x) \cdot z'$ and $y = p(y) \cdot z'$.
Thus $\theta (x) = p(x) \cdot \alpha (z')$ and $\theta (y) = p(y) \cdot \alpha (z')$.
It follows that $\theta (x) \vee \theta (y) = p(x \vee y) \cdot \alpha (z')$.
But $p(x \vee y) \cdot \alpha (z') = \theta (x \vee y)$.

We now prove that $A$ is essentially surjective.

Let $Y$ be a $[B|M]$-set.
If $Y$ is empty then it gives rise to the Boolean supported action of $S$ on the empty set.
In what follows, therefore, we shall assume that $Y$ is non-empty
Define
$$X = \bigsqcup_{e \in \mathsf{Proj}(S)} Y/\equiv_{e},$$
a disjoint union,
and define
$p \colon X \rightarrow \mathsf{Proj}(S)$
by $p([x]_{e}) = e$.
Define 
$$s \bullet [x]_{e} = [\hat{s} \cdot x]_{(se)^{+}}.$$
This defines a supported action of $S$ on $X$, but we now need to show that the axioms
(E3)--(E6) hold for $X$.

(E3) and (E4) hold: 
Put $0 = [y]_{0}$ where $y \in Y$ is any element.
By (MPA7), it does not matter which element $y$ we choose since all elements of $Y$ are identified.
This is the minimum element of the poset $X$ by Remark~\ref{rem:agatha}.
The definition of the action then shows that $0$, as defined, is fixed.

(E5) holds.
We need a preliminary result first:
let $x,y \in Y$ such that $x \equiv_{ef} y$.
We prove that there exists a $w \in Y$ such that
$w \equiv_{e} x$ and $w \equiv_{f} y$,
where such a $w$ is unique up to $\equiv_{e\vee f}$-equivalence.
Now to the proof of our preliminary result.
By (MPA9)
there is a $w \in Y$ such that $w \equiv_{e} x$ and $w \equiv_{\bar{e}} y$.
In a Boolean algebra, we have that $f = ef \vee \bar{e}f$.
By (MPA4),
we have that $w \equiv_{ef} x$ and $w \equiv_{\bar{e}f} y$. 
It follows that $w \equiv_{ef} y$ and $w \equiv_{\bar{e}f} y$.
By (MPA8), we have that $w \equiv_{f} y$.
Now, suppose that there exists a $w' \in Y$ such that
$w' \equiv_{e} x$ and $w' \equiv_{f} y$.
Then by (MPA8), we have that $w \equiv_{e \vee f} w'$.

We can now prove that axiom (E5) holds.
Suppose that  $[x]_{e} \approx [y]_{f}$.
We prove that
$[w]_{e \vee f} = [x]_{e} \vee [y]_{f}$
where $w \in Y$ is any element such that 
$w \equiv_{e} x$ and $w \equiv_{f} y$.
Observe that $x \equiv_{ef} y$ means precisely $[x]_{e} \approx [y]_{f}$.
Thus if $[x]_{e} \approx [y]_{f}$, then, using the element $w$ guaranteed above,
we have that $[x]_{e}, [y]_{f} \leq [w]_{e \vee f}$.
Suppose, now,  that $[x]_{e}, [y]_{f} \leq [w']_{i}$.
Then $e \leq i$ and $x \equiv_{e} w'$, and $f \leq i$ and $y \equiv_{f} w'$.
Thus $e \vee  f \leq i$, and $w \equiv_{e} w'$ and $w \equiv_{f} w'$.
It follows, by  (MPA8), that $w' \equiv_{e \vee f} w$.
By Remark~\ref{rem:agatha}, we have therefore shown that
$[w]_{e \vee f} \leq [w']_{i}$.
We have proved that $[w]_{e \vee f} = [x]_{e} \vee [y]_{f}$.

(E6) holds.
We show that if $[x]_{e} \approx [y]_{f}$ then 
$s \bullet ([x]_{e} \vee [y]_{f}) = s \bullet [x]_{e} \vee s \bullet [y]_{f}$.
Let $s = s^{+}\hat{s}$.
Then $s \bullet ([x]_{e} \vee [y]_{f}) = [\hat{s} \cdot z]_{(s(e \vee f))^{+}}$.
Also, $s \bullet [x]_{e} = [\hat{s} \cdot x]_{(se)^{+}}$ and $s \bullet [y]_{f} = [\hat{s} \cdot y]_{(sf)^{+}}$.
From $z \equiv_{e} x$ and $z \equiv_{f} y$ and using (MPA6),
we have that $\hat{s} \cdot z \equiv_{\hat{s} \ast e} \hat{s} \cdot x$ and $\hat{s} \cdot z \equiv_{\hat{s} \ast f} \hat{s} \cdot y$.
Observe that $(se)^{+} \leq (\hat{s}e)^{+}$ and $(sf)^{+} \leq (\hat{s}f)^{+}$.
Thus by (MPA4), we have that 
$\hat{s} \cdot z \equiv_{(se)^{+}} \hat{s} \cdot x$ and $\hat{s} \cdot z \equiv_{(sf)^{+}} \hat{s} \cdot y$.
The result now follows by the next calculation 
$$(s(e\vee f))^{+} = (se \vee sf)^{+} = (se)^{+} \vee (sf)^{+}.$$

We have therefore defined a supported action $(S,X,p)$.

The proof that the functor is essentially surjective now follows from the proof of the same in
Theorem~\ref{them:mary}.
\end{proof}

In the light of Theorem~\ref{them:marytwo} and the work of Garner \cite{G2024}, we have the following.

\begin{theorem} Let $\mathcal{V}$ be a non-degenerate finitary Cartesian closed variety.
Then there is a Boolean left restriction monoid $S$  such that 
$\mathcal{V}$ is equivalent to $S$-$\mbox{\bf FBSup}$.
\end{theorem}

\begin{remark}{\em The fact that the category $S$-{\bf FBSup} is Cartesian closed,
of course, follows from Theorem~\ref{them:marytwo}.
However, we proved Theorem~\ref{them:CCV2} because we 
wanted an explicit description of the exponential object.}
\end{remark}

\begin{remark}{\em In \cite{G2024, G2025}, Garner {\em constructs} the monoid $S$.
The obvious question is whether $S$ arises {\em directly} from $\mathcal{V}$.}
\end{remark}



\begin{thebibliography}{99}


\bibitem{BW} M. Barr, Ch. Wells, {\em Category theory for computing science}, Prentice Hall, 1990.


\bibitem{CG} R. Cockett, R. Garner, Generalising the \'etale groupoid-complete pseudogroup correspondence,
{\em Adv. Math.} {\bf 392} (2021), 108030.

\bibitem{EF} D. Easdown, J. East, D. G. FitzGerald, Presentations of factorizable inverse monoids, 
{\em Acta Sci. Math. (Szeged)} {\bf 71} (2005), 509--520.

\bibitem{F2010} J. Funk, B. Steinberg, The universal covering of an inverse semigroup,
{\em Appl. Categ. Structures} {\bf 18} (2010), 135--163.

\bibitem{G2024} R. Garner, Cartesian closed varieties I: the classification theorem,
{\em Algebra Univers.} {\bf 85}, 38 (2024). \url{https://doi.org/10.1007/s00012-024-00869-1}.

\bibitem{G2025} R. Garner, Cartesian closed varieties II: links to operator algebra, arXiv:2302.04403.

\bibitem{Higman} G. Higman, Finitely presented infinite simple groups, Notes on Pure Mathematics., ANU, Canberra, Australia, 1974.

\bibitem{JS} M. Jackson, T. Stokes, Modal restriction semigroups: towards an algebra of functions,
{\em Int. J. Algebra Comput.} {\bf 21} (2011), 1053--1095.


\bibitem{L1988} J. Lambek, P. J. Scott, {\em Introduction to higher order categorical logic}, CUP, Cambridge, 1986.

\bibitem{L2025} M. V. Lawson, The structure of \'etale Boolean right restriction monoids, {\em J. Aust. Math. Soc.} {\bf 121} (2026), 94--125.

\bibitem{Lawson} M. V. Lawson, {\em Inverse semigroups: the theory of partial symmetries}, Second Edition, World Scientific, Singapore, 2026.

\bibitem{LR2021} M. V. Lawson, P. Resende, Morita equivalence of pseudogroups, {\em J. Algebra} {\bf 586} (2021), 718--755.

\bibitem{CWM} S. Mac\,Lane, {\em Categories for the working mathematician}, Second Edition, Springer, 1997. 

\bibitem{Statman} R. Statman, Cartesian monoids, {\em Electron. Notes Theor. Comput. Sci.} {\bf 265} (2010), 437--451.

\bibitem{S2011} B. Steinberg, Strong Morita equivalence of inverse semigroups,
{\em Houston J. Math.} {\bf 37} (2011), 895--927.

\end{thebibliography}
\end{document}